\input ./style/arxiv-general.cfg
\documentclass[aap,MSNbibl,secfloat,seceqn,dvips]{arximspdf}
\makeatletter
   \@ifpackageloaded{graphicx}{}{\usepackage{graphicx}}
\makeatother
\usepackage{mathrsfs}
\usepackage{algpseudocode}
\usepackage{algorithm}

\doi{10.1214/15-AAP1145}
\volume{26}
\issue{4}
\pubyear{2016}
\firstpage{2211}
\lastpage{2256}
\docsubty{FLA}

\makeatletter

\newcommand{\rrvert}{\vert}

\newcommand{\llvert}{\vert}
\def\sfrac#1#2{#1/#2}

\newcommand{\eqref}[1]{(\ref{#1})}
\newcommand{\R}{\mathbb{R}}
\newcommand{\N}{\mathbb{N}}
\newcommand{\C}{\mathbb{C}}
\newcommand{\E}{\mathbb{E}}
\newcommand{\calG}{\mathcal{G}}
\newcommand{\eps}{\varepsilon}
\let\Pr\relax
\newcommand{\Pr}{\operatorname{Pr}}
\newcommand{\sgn}{\operatorname{sgn}}
\newcommand{\Var}{\operatorname{Var}}
\newcommand{\Binom}{\operatorname{Binom}}
\newcommand{\Pois}{\operatorname{Pois}}
\newcommand{\symdiff}{\Delta}

\newtheorem{theorem}{Theorem}[section]
\newproclaim{assumption}{Assumption}[section]
\newtheorem{lemma}[theorem]{Lemma}
\newtheorem{proposition}[theorem]{Proposition}
\newproclaim{remark}[theorem]{Remark}
\newproclaim{definition}[theorem]{Definition}

\renewcommand{\C}{\mathcal{C}}
\newcommand{\cond}{\mathscr{C}_{\mathrm{eff}}}
\newcommand{\res}{\mathscr{R}_{\mathrm{eff}}}
\newcommand{\acc}{\operatorname{acc}}
\newcommand{\hypergeom}{\operatorname{HyperGeom}}
\makeatother

\begin{document}
\begin{frontmatter}

\title{Belief propagation, robust reconstruction and optimal recovery
of block models}
\runtitle{Optimal recovery of block models}

\begin{aug}
\author[A]{\fnms{Elchanan}~\snm{Mossel}\thanksref{M1,M2,T1}},
\author[B]{\fnms{Joe}~\snm{Neeman}\corref{}\thanksref{M3,M4,T1}\ead[label=e2]{joeneeman@gmail.com}}
\and
\author[C]{\fnms{Allan}~\snm{Sly}\thanksref{M2,M5,T2}}
\runauthor{E. Mossel, J. Neeman and A. Sly}
\affiliation{University of Pennsylvania\thanksmark{M1},
University of California, Berkeley\thanksmark{M2},\\
University of Bonn\thanksmark{M3},
University of Texas, Austin\thanksmark{M4}\\ and
Australian National University\thanksmark{M5}}
\address[A]{E. Mossel\\
Department of Statistics\\
The Wharton School\\
University of Pennsylvania\\
400 Jon M. Huntsman Hall\\
3730 Walnut Street\\
Philadelphia, Pennsylvania 19104\\
USA\\
and\\
Department of Statistics\\
University of California, Berkeley\\
367 Evans Hall\\
Berkeley, California 94720\\
USA}
\address[B]{J. Neeman\\
Hausdorff Center for Mathematics\\
University of Bonn\\
Endenicher Allee 60\\
D-53115 Bonn\\
Germany\\
and\\
Department of Mathematics\\
University of Texas, Austin\\
RLM 8.100\\
2515 Speedway Stop C1200\\
Austin, Texas 78712\\
USA\\
\printead{e2}}
\address[C]{A. Sly\\
Department of Statistics\\
University of California, Berkeley\\
367 Evans Hall\\
Berkeley, California 94720\\
USA\\
and\\
Mathematical Sciences Institute\\
The Australian National University\\
John Dedman Building 27\\
Union Lane\\
Canberra, ACT 2601\\
Australia}
\end{aug}
\thankstext{T1}{Supported by NSF Grant DMS-11-06999 and DOD ONR Grants
N000141110140 and N00014-14-1-0823 and Grant 328025 from the Simons Foundation.}
\thankstext{T2}{Supported by an Alfred Sloan Fellowship and NSF Grant
DMS-12-08338.}

%
\received{\smonth{9} \syear{2014}}
%
\revised{\smonth{4} \syear{2015}}

%
\begin{abstract}
We consider the problem of reconstructing sparse symmetric block models
with two blocks
and connection probabilities $a/n$ and $b/n$ for inter- and intra-block
edge probabilities, respectively.
It was recently shown that one can do better than a random guess if and
only if $(a-b)^2 > 2(a+b)$.
Using a variant of belief propagation, we give a reconstruction
algorithm that is \emph{optimal} in the sense
that if $(a-b)^2 > C (a+b)$ for some constant $C$ then our algorithm
maximizes the
fraction of the nodes labeled correctly. Ours is the only algorithm
proven to
achieve the optimal fraction of nodes labeled correctly.
Along the way, we prove some results of independent interest regarding
{\em robust reconstruction} for the Ising model on regular and Poisson trees.
\end{abstract}

%
\end{frontmatter}

\setcounter{footnote}{2}

\section{Introduction}\label{sec1}

\subsection{Sparse stochastic block models}
Stochastic block models were introduced more than 30 years ago~\cite{HLL83}
in order to study the problem of community detection in random graphs.
In these models, the nodes in a graph are divided into two or more communities,
and then the edges of the graph are drawn independently at random, with
probabilities depending on which communities the edge lies between.
In its simplest incarnation---which we will study here---the model
has $n$ vertices divided into
two classes of approximately equal size, and two parameters:
$a/n$ is the probability that each within-class edge will appear,
and $b/n$ is the probability that each between-class edge will appear.
Since their \hyperref[sec1]{Introduction}, a large body of literature has been written about
stochastic block models,
and a multitude of efficient algorithms have been developed for the
problem of inferring
the underlying communities from the graph structure. To name a few, we now
have algorithms based on maximum-likelihood methods~\cite{SN97}, belief
propagation~\cite{Z2011}, spectral methods~\cite{M01}, modularity
maximization~\cite{BC09}
and a number of combinatorial methods~\cite{BCLS87,DF89,JS98,CK01}.

Early work on the stochastic block model mainly focused on fairly dense
graphs: Dyer and Frieze~\cite{DF89}; Snijders and Nowicki~\cite{SN97};
and Condon and Karp~\cite{CK01} all gave algorithms that
will correctly recover the exact communities in a graph from the stochastic
block model, but only when $a$ and $b$ are polynomial in $n$.
In a substantial improvement, McSherry~\cite{M01} gave a spectral algorithm
that succeeds when $a$ and $b$ are logarithmic in $n$; this had been
anticipated previously by Boppana~\cite{B87}, but his proof was
incomplete. McSherry's parameter range was later
equalled by Bickel and Chen~\cite{BC09} using an algorithm based on
modularity maximization.

We also note that related but different problems of planted coloring
were studied in
Blum and Spencer~\cite{BlumSpencer95} in the dense case, and Alon and
Kahale~\cite{AlonKahale97} in the sparse case.

The $O(\log n)$ barrier is important because if the average degree of a
block model is logarithmic or larger, it is possible to exactly recover
the communities with high probability as $n \to\infty$. On the other hand,
if the average degree is less than logarithmic then some fairly straightforward
probabilistic arguments show that it is not possible to completely recover
the communities. When the average degree is constant, as it will be
in this work, then one cannot get more than a constant fraction of the labels
correct.

Despite these apparent difficulties, there are important
practical reasons for considering block models with constant average degree.
Indeed, many real networks are very sparse.
For example, Leskovec et~al. \cite{LLDM08}
and Strogatz~\cite{Strogatz01}
collected and studied a vast collection of large network
datasets, many of which had millions of nodes,
but most of which had an average degree of no more than 20; for instance,
the LinkedIn network studied by Leskovec et~al. had approximately
seven million nodes, but only 30 million edges.
Moreover, the very fact that sparse block models are impossible to infer
exactly may be taken as an argument for studying them: in real networks
one does not
expect to recover the communities with perfect accuracy, and so it makes
sense to study models in which this is not possible either.

Although sparse graphs are immensely important, there is not
yet much known about very sparse stochastic block models.
In particular, there is a gap between what is known for block models
with a constant
average degree and those with an average degree that grows with the
size of
the graph.
Until recently, there was only one algorithm---due to~\cite{CO10}, and
based on spectral methods---which was guaranteed to do anything at all
in the
constant-degree regime,
in the sense that it produced communities which have a better-than-50\% overlap
with the true communities.

Despite the lack of rigorous results, a beautiful conjectural picture has
recently emerged, supported by simulations and deep but nonrigorous
physical intuition. We are referring specifically to work of Decelle
et~al. \cite{Z2011}, who conjectured the existence of a threshold,
below which
is it not possible to find the communities better than by guessing randomly.
In the case of two communities of equal size, they pinpointed the location
of the conjectured threshold.
This threshold has since been rigorously confirmed;
a sharp lower bound on its location was given by the authors~\cite
{MoNeSl13}, while sharp
upper bounds were given independently by Massouli\'e~\cite{Massoulie13}
and by the authors~\cite{MoNeSl14}.

\subsection{Our results: Optimal reconstruction}
Given that it is not possible to completely recover the communities
in a sparse block model, it is natural to ask how accurately one may
recover them. In~\cite{MoNeSl13}, we gave an upper bound on the recovery
accuracy; here, we will show that that bound is tight---at least, when the
signal to noise ratio is sufficiently high---by giving an algorithm which
performs as well as the upper bound. Our main result may be stated
informally as follows.\footnote{An extended abstract stating the
results of the current paper~\cite{MoNeSl14b} appeared in the
proceedings of COLT 2014 (where it won the best paper award).}

%
\begin{theorem}\label{teo:intro}
Let $p_G(a,b)$ be the highest asymptotic accuracy that any algorithm
can achieve
in reconstructing communities of the block model with parameters $a$
and $b$. We provide an algorithm that achieves accuracy of $p_G(a,b)$
with probability tending to 1 as $n \to\infty$, provided that $(a-b)^2/(a+b)$
is sufficiently large.
\end{theorem}

To put Theorem~\ref{teo:intro} into the context of
earlier work~\cite{MoNeSl13,MoNeSl14,Massoulie13} by the authors
and Massouli\'e, those works showed that $p_G(a,b) > 1/2$ if and only if
$(a-b)^2 > 2(a+b)$; in the case that $p_G(a,b) > 1/2$, they also provided
algorithms whose accuracy was bounded away from $1/2$. However, those
algorithms were not guaranteed (and are not expected) to have \emph
{optimal} accuracy, only
\emph{nontrivial} accuracy. In other words, previous results have
shown that for every value of $a,b$ such that $(a-b)^2 > 2(a+b)$ there
exists an algorithm that recovers (with high probability) a fraction
$q(a,b) > 1/2$ of the nodes correctly. Our results provide an algorithm
that [when $(a-b)^2 > C(a+b)$ for a large constant $C$] recovers the
optimal fraction of nodes $p_G(a,b)$ in the sense that it is
information theoretically impossible for any other algorithms to
recover a bigger fraction.

Our new algorithm, which is based on
belief propagation, is essentially an algorithm for locally improving an
initial guess at the communities. In our current analysis, the initial guess
is provided by a previous algorithm of the authors~\cite{MoNeSl14}, which
we use as a black box.
We should mention that standard
belief propagation with random uniform initial messages and without our
modifications and also without a good initial
guess, is also conjectured to have optimal accuracy~\cite{Z2011}.
However, at the moment, we do not know of any approach to analyze the
vanilla version of BP for this problem.

As a major part of our analysis, we prove a result about broadcast
processes on trees
that may be of independent interest. Specifically, we prove that if the
signal-to-noise ratio of the broadcast process is sufficiently high, then
adding extra noise at the leaves of a large tree does not hurt our
ability to
guess the label of the root given the labels of the leaves.
In other words, we show that for a certain model on trees,
belief propagation initialized with arbitrarily noisy messages
converges to the optimal solution as the height of the tree
tends to infinity.
We prove our result
for regular trees and Galton--Watson trees with Poisson offspring, but we
conjecture that it also holds for general trees, and even if the
signal-to-noise ratio is low.

We should point out that spectral algorithms---which, due to their efficiency,
are very popular algorithms for this model---empirically do not
perform as well
as BP on very sparse graphs (see, e.g.,~\cite{KMMNSZZ13}). This is
despite the
recent appearance of two new spectral algorithms, due to~\cite{KMMNSZZ13}
and~\cite{Massoulie13}, which were specifically designed for clustering
sparse block models. The algorithm of~\cite{KMMNSZZ13} is
particularly relevant
here, because it was derived by linearizing belief propagation; empirically,
it performs well all the way to the impossibility threshold, although
not quite as well as BP. Intuitively, the linear aspects of spectral algorithms
(i.e., the fact that they can be implemented---via the power method---using
local linear updates) explain why they cannot achieve optimal performance.
Indeed, since the optimal local updates (those given by BP) are nonlinear,
any method based on linear updates will be suboptimal.

\subsection{Dramatis personae}

Before defining everything carefully, we briefly introduce the three main
objects and their relationships.
\begin{itemize}
\item The \emph{block model detection problem} is the problem
of detecting communities in a sparse stochastic block model.
\item In the \emph{tree reconstruction problem}, there is a two-color
branching process in which every node has some children of its own color
and some children of the other color. We observe the family tree
of this process and also all of the colors in some generation; the goal is
to guess the color of the original node.
\item The \emph{robust tree reconstruction problem} is like the tree
reconstruction problem, except that instead of observing exactly the
colors in some generation, our observations contain some noise.
\end{itemize}
The two tree problems are related to the block model problem
because a neighborhood in the stochastic block model
looks like a random tree from one of the tree problems.
This connection was proved in~\cite{MoNeSl13}, who also showed that
tree reconstruction is ``easier'' than the block model detection
(in a sense that we will make precise later). The current work
has two main steps: we show that block model detection is
``easier'' than robust tree reconstruction, and we show that---for a certain range of parameters---robust tree reconstruction
is exactly as hard as tree reconstruction.

\section{Definitions and main results}

\subsection{The block model}

In this article, we restrict the stochastic block model to the case
of two classes with roughly equal size.

%
%
\begin{definition}[(Stochastic block model)]
The {\em block model} on $n$ nodes
is constructed by first labeling each node $+$ or $-$ with equal
probability independently.
Then each edge is included in the graph independently, with probability
$a/n$ if its endpoints have the same label and $b/n$ otherwise. Here,
$a$ and $b$ are two positive parameters.
We write $\calG(n, a/n, b/n)$ for this distribution of (labeled) graphs.
\end{definition}
For us, $a$ and $b$ will be fixed,
while $n$ tends to infinity. More generally, one may consider the case
where $a$ and $b$ may be allowed to
grow with $n$.
As conjectured by~\cite{Z2011},
the relationship between $(a-b)^2$ and $(a+b)$ turns out to be of
critical importance for
the reconstructability of the block model.
%
%
\begin{theorem}[(Threshold for nontrivial detection~\cite
{Massoulie13,MoNeSl13,MoNeSl14})]
For the block model with parameters $a$ and $b$, it holds that:
\begin{itemize}
\item
If $(a - b)^2 < 2(a+b)$ then
the node labels cannot be inferred from the unlabeled graph with better
than $50\%$ accuracy (which could also be done just by random guessing).
\item If $(a - b)^2 > 2(a+b)$ then it is
possible to infer the labels with better than $50\%$ accuracy.
\end{itemize}
\end{theorem}

\subsection{Broadcasting on trees}
Our study of optimal reconstruction accuracy is based on the local structure
of $\calG(n, a/n, b/n)$, which requires the notion of the {\em
broadcast process on a tree}.

Consider an infinite, rooted tree. We will identify such a tree $T$
with a
subset of $\N^*$, the set of finite strings of natural numbers,
with the property that if
$v \in T$ then any prefix of $v$ is also in $T$. In this way, the root
of the
tree is naturally identified with the empty string, which we will denote
by $\rho$.
We will write $uv$ for the concatenation of the strings $u$ and $v$, and
$L_k(u)$ for the $k$th-level descendents of~$u$; that is,
$L_k(u) = \{uv \in T: \llvert v\rrvert = k\}$. Also, we will write
$\C(u) \subset\N$ for the indices of $u$'s children relative to itself,
that is, $i \in\C(u)$ if and only if $ui \in L_1(u)$.

%
\begin{definition}[(Broadcast process on a tree)]
Given a parameter $\eta\ne1/2$ in $[0, 1]$ and a tree $T$, the {\em
broadcast process on $T$} is
a two-state Markov process $\{\sigma_u: u \in T\}$ defined as follows:
let $\sigma_\rho$ be $+$ or $-$ with probability $\frac{1}{2}$.
Then, for each $u$ such that $\sigma_u$ is defined,
independently for every $v \in L_1(u)$ let $\sigma_v = \sigma_u$ with
probability
$1-\eta$ and $\sigma_v = -\sigma_\rho$ otherwise.
\end{definition}

This broadcast process has been extensively studied, where the major
question is
whether the labels of vertices far from the root of the tree give
any information on the label of the root.
For general trees, this question was answered definitively by Evans
et~al. \cite{EvKePeSc00}, after many other contributions including~\cite
{KestenStigum66,BlRuZa95}.
The complete statement of the theorem requires the notion of \emph
{branching number},
which we would prefer not to define here (see~\cite{EvKePeSc00}).
For our purposes, it suffices to know that a
$d$-ary tree has branching number $d$ and that a Poisson branching
process tree
with mean $d > 1$ has branching number $d$ (almost surely, and
conditioned on nonextinction).
%
%
\begin{theorem}[(Tree reconstruction threshold~\cite{EvKePeSc00})]
\label{teo:tree-threshold}
Let $\theta= 1-2\eta$ and $d$ be the branching number of $T$.
Then
\[
\E\bigl[\sigma_{\rho} \mid\sigma_u: u \in
L_k( \rho) \bigr] \to0
\]
in probability as $k \to\infty$ if and only if $d \theta^2 \leq1$.
\end{theorem}
The theorem implies in particular that if $d \theta^2 > 1$ then for
every $k$ there is an algorithm which guesses
$\sigma_\rho$ given $\sigma_{L_k(\rho)}$, and which succeeds with
probability
bounded away from $1/2$. If $d \theta^2 \leq1$ there is no such algorithm.

\subsection{Robust reconstruction on trees}

Janson and Mossel~\cite{JansonMossel04} considered a version of the
tree broadcast process
that has extra noise at the leaves.
%
%
\begin{definition}[(Noisy broadcast process on a tree)]
Given a broadcast process $\sigma$ on a tree $T$ and a parameter
$\delta\in[0, 1/2)$, the
{\em noisy broadcast process on $T$} is the process $\{\tau_u: u \in
T\}$ defined
by independently taking $\tau_u = -\sigma_u$ with probability $\delta
$ and $\tau_u = \sigma_u$
otherwise.
\end{definition}

We observe that the noise present in $\sigma$ and the noise present in
$\tau$ have
qualitatively different roles, since the noise present in $\sigma$
propagates down the tree
while the noise present in $\tau$ does not. Janson and Mossel~\cite
{JansonMossel04} showed
that the range of parameters for which $\sigma_\rho$ may be
nontrivially reconstructed from $\sigma_{L_k}$ is
the same as the range for which $\sigma_\rho$ may be nontrivially
reconstructed from $\tau_{L_k}$. In other
words, additional noise at the leaves has no effect on whether the
root's signal
propagates arbitrarily far. One of our main results is a quantitative
version of this statement
(Theorem~\ref{teo:robust-tree-intro}): we show that for a certain
range of
parameters, the presence of noise at the leaves does not even affect
the accuracy with
which the root can be reconstructed.

\subsection{The block model and broadcasting on trees}
The connection between the community reconstruction problem on a graph
and the root reconstruction problem on a tree was first pointed out
in~\cite{Z2011} and made rigorous in~\cite{MoNeSl13}. The basic
idea is the following:
\begin{itemize}
\item
A neighborhood in $G$ looks like a Galton--Watson tree with offspring
distribution $\Pois((a+b)/2)$ [which almost surely has branching number
$d = (a + b)/2$].
\item
The labels on the neighborhood look as
though they came from a broadcast process with parameter $\eta= \frac
{b}{a+b}$.
\item
With these parameters, $\theta^2 d = \frac{(a-b)^2}{2(a+b)}$, and so
the conjectured threshold for community reconstruction is the same as the
proven threshold for tree reconstruction.
\end{itemize}
This local approximation can be formalized as convergence locally on
average, a~type of local weak convergence defined in~\cite{MMS12}.
We should mention that in the case
of more than two communities (i.e., in the case that the broadcast process
has more than two states) then the picture becomes rather more complicated,
and much less is known; see~\cite{Z2011,MoNeSl13} for some conjectures.

\subsection{Reconstruction probabilities on trees and graphs}
Note that Theorem~\ref{teo:tree-threshold} only answers the question
of whether one can achieve asymptotic reconstruction accuracy better
than $1/2$.
Here, we will be interested in more detailed information about the
actual accuracy of reconstruction,
both on trees and on graphs.

Note that in the tree reconstruction problem, the optimal estimator of
$\sigma_\rho$ given~$\sigma_{L_k(\rho)}$ is easy to write down: it is
simply the sign of
$X_{\rho,k}:= 2 \Pr(\sigma_\rho= + \mid\sigma_{L_k(\rho)}) - 1$.
Compared to the trivial procedure of guessing $\sigma_\rho$
completely at random,
this estimator has an expected gain of
\[
\E\bigl\llvert\Pr(\sigma_\rho= + \mid\sigma_{L_k(\rho)}) -
\tfrac{1}{2} \bigr\rrvert= \tfrac{1}{2} \E\bigl[\bigl\llvert\E[
\sigma_{\rho} \mid\sigma_{L_k(\rho)}]\bigr\rrvert\bigr].
\]
It is now natural to define:
%
%
\begin{definition}[(Tree reconstruction accuracy)]\label{def:tree-accuracy}
Let $T$ be an infinite Galton--Watson tree with $\Pois((a+b)/2)$
offspring distribution, and
$\eta= \frac{b}{a+b}$. Consider the broadcast process on the tree
with parameter $\eta$ and define
%
%
\begin{equation}
\label{eq:pT-def} p_T(a,b) = \frac{1}2 + \lim
_{k\to\infty} \E\biggl\llvert\Pr(\sigma_\rho=+\mid
\sigma_{L_k(\rho)}) - \frac{1}2 \biggr\rrvert.
\end{equation}
In words, $p_T(a,b)$ is the probability of correctly inferring $\sigma
_\rho$ given the
``labels at infinity''.
\end{definition}

Note that by Theorem~\ref{teo:tree-threshold},
$p_T(a,b) > 1/2$ if and only if $(a-b)^2 > 2(a+b)$.

We remark that the limit in Definition~\ref{def:tree-accuracy}
always exists because the right-hand side
is nonincreasing in $k$. To see this, it helps to write $p_T(a,b)$ in
a different way: let $\mu_k^+$ be the distribution of $\sigma
_{L_k(\rho)}$ given
$\sigma_\rho= +$ and let $\mu_k^-$ be the distribution of $\sigma
_{L_k(\rho)}$
given $\sigma_\rho= -$. Then
\[
\E\bigl\llvert\Pr(\sigma_\rho=+\mid\sigma_{L_k(\rho)}) -
\tfrac{1}2 \bigr\rrvert= \tfrac{1}2 d_{\mathrm{TV}} \bigl(
\mu_k^+,\mu_k^- \bigr),
\]
where $d_{\mathrm{TV}}$ denotes the total variation distance.
Next, note that since labels at levels
$k' > k$ are independent of $\sigma_\rho$ given $\sigma_{L_k(\rho)}$,
\[
\Pr(\sigma_\rho=+\mid\sigma_{L_k(\rho)}) = \Pr(
\sigma_\rho=+\mid\sigma_{L_k(\rho)}, \sigma_{L_{k+1}(\rho
)},
\sigma_{L_{k+2}(\rho)}, \dots).
\]
Hence, if we set $\nu_k^+$ to be the distribution of $\{\sigma
_{L_{k'}}(\rho): k' \ge k\}$
and similarly for~$\nu_k^-$, we have
\[
\E\bigl\llvert\Pr(\sigma_\rho=+\mid\sigma_{L_k(\rho)}) -
\tfrac{1}2 \bigr\rrvert= \tfrac{1}2 d_{\mathrm{TV}} \bigl(
\nu_k^+,\nu_k^- \bigr).
\]
Now the right-hand side is clearly nonincreasing in $k$, because $\nu
_{k+1}^+$ can
be obtained from $\nu_k$ by marginalization.

One of the main results of~\cite{MoNeSl13} is that the graph reconstruction
problem is at least as hard as the tree reconstruction problem in the
sense that for
any community-detection algorithm, the asymptotic accuracy of that algorithm
is bounded by $p_T(a,b)$.

%
\begin{definition}[(Graph reconstruction accuracy)]
Let $(G, \sigma)$ be a labeled graph on $n$ nodes.
If $f$ is a function that takes a graph and returns a labeling of it,
we write
\[
\acc(f, G, \sigma) = \frac{1}{2} + \biggl\llvert\frac{1}{n} \sum
_{v} 1 \bigl( \bigl(f(G) \bigr)_v =
\sigma_v \bigr) - \frac{1}{2} \biggr\rrvert
\]
for the accuracy of $f$ in recovering the labels $\sigma$.
For $\varepsilon> 0$, let
\[
p_{G,n,\varepsilon}(a,b) = \sup_f \sup\bigl\{p: \Pr\bigl(
\acc(f, G, \sigma) \ge p \bigr) \ge\varepsilon\bigr\},
\]
where the first supremum ranges over all functions $f$, and the
probability is taken
over $(G, \sigma) \sim\calG(n, a/n, b/n)$.
Let
\[
p_G(a,b) = \lim_{\varepsilon\to0} \limsup
_{n \to\infty} p_{G,n,\varepsilon}(a,b),
\]
where the limit exists because $p_{G,n,\varepsilon}(a,b)$ is monotonic in
$\varepsilon$.
\end{definition}
One should think of $p_{G}(a,b)$ as the optimal fraction of nodes that
can be reconstructed correctly by any algorithm
(not necessarily efficient) that only gets to observe an unlabeled graph.
More precisely, for any algorithm and any $p > p_G(a,b)$, the
algorithm's probability of achieving
accuracy $p$ or higher converges to zero as $n$ grows.
Note that the symmetry between the $+$ and $-$ is reflected in the
definition of $\acc$ (e.g., in the
appearance of the constant $1/2$), and also that $\acc$ is defined to
be large if $f$ gets most labels \emph{incorrect}
(because there is no way for an algorithm to break the symmetry between
$+$ and $-$).

An immediate corollary of the analysis of~\cite{MoNeSl13} implies
that graph reconstruction is
always at most as accurate as tree reconstruction.

%
\begin{theorem}[(Graph detection is harder than tree
reconstruction~\cite{MoNeSl13})]\label{teo:mns:13}
\[
p_G(a,b) \leq p_T(a,b).
\]
\end{theorem}
We remark that Theorem~\ref{teo:mns:13} is not stated explicitly
in~\cite{MoNeSl13}; because
the authors
were only interested in the case $(a-b)^2 \le2(a+b)$, the claimed
result was that
$(a-b)^2 \le2(a+b)$ implies $p_G(a,b) = \frac{1}2$. However, a cursory
examination of the
proof of~\cite{MoNeSl13}, Theorem 1, reveals that
the claim was proven in two stages: first, they prove via a coupling argument
that $p_G(a,b) \le p_T(a,b)$ and then they apply
Theorem~\ref{teo:tree-threshold} to show that $(a-b)^2 \le2(a+b)$
implies $p_T(a,b) = \frac{1}2$.

\subsection{Our results}

In this paper, we consider the high signal-to-noise case, namely the
case that $(a-b)^2$ is significantly larger than $2(a+b)$. In this
regime, we
give an algorithm (Algorithm~\ref{alg}) which achieves an accuracy of
$p_T(a,b)$.

%
\begin{theorem}\label{teo:alg-is-optimal}
There exists a constant $C$ such that if
$(a-b)^2 \ge C (a + b)$ then
\[
p_G(a,b) = p_T(a,b).
\]
Moreover, there is a polynomial time algorithm such that for all such
$a,b$ and every $\eps> 0$,
with probability tending to one as $n \to\infty$,
the algorithm reconstructs the labels with accuracy $p_G(a,b) - \eps$.
\end{theorem}

We will assume for simplicity that our algorithm is given the
parameters $a$ and~$b$.
This is a minor assumption because $a$ and $b$ can be estimated from
the data
to arbitrary accuracy~\cite{MoNeSl13}, Theorem 3.

A key ingredient of Theorem~\ref{teo:alg-is-optimal}'s proof is a procedure
for amplifying a clustering that is a slightly better than a random guess
to obtain optimal clustering.
In order to discuss this procedure, we define the problem of ``robust
reconstruction'' on trees.

%
\begin{definition}[(Robust tree reconstruction accuracy)]\label
{def:robust-tree-accuracy}
Consider the noisy tree broadcast process with\vspace*{1pt} parameters $\eta= \frac
{a}{a+b}$
and $\delta\in[0, 1/2)$ on a Galton--Watson
tree with offspring distribution $\Pois((a+b)/2)$.
We define the robust reconstruction accuracy as
\[
\tilde p_T(a,b) = \frac{1}2 + \liminf_{\delta\to1/2}
\liminf_{k\to
\infty} \E\biggl\llvert\Pr(\sigma_\rho=+\mid
\tau_{L_k(\rho)}) - \frac{1}2 \biggr\rrvert.
\]
\end{definition}

Our main technical result is that when $a-b$ is large enough then in
fact the extra noise does not have any effect
on the reconstruction probability.

%
\begin{theorem} \label{teo:robust-tree-intro}
There exists a constant $C$ such that if
$(a-b)^2 \ge C (a + b)$ then
\[
\tilde p_T(a,b) = p_T(a,b).
\]
\end{theorem}

We conjecture that the
robust reconstruction accuracy is independent of $\delta$ for any parameters,
and also for more general trees; however, our proof does not naturally
extend to
cover these cases.

\subsection{Algorithmic amplification and robust reconstruction}

The second\break main ingredient in Theorem~\ref{teo:alg-is-optimal}
connects the community detection problem to
the robust tree reconstruction problem: we show that given a suitable
algorithm for providing a better-than-random initial guess at the
communities, the community
detection problem is easier than the robust reconstruction problem, in
the sense that one can achieve an accuracy of $\tilde p_T(a,b)$.

%
\begin{theorem}\label{teo:alg-gets-ptilde}
For all $a$ and $b$, $p_G(a,b) \ge\tilde p_T(a,b)$.
Moreover, there is a polynomial time algorithm such that for all such
$a,b$ and every $\eps> 0$,
with probability tending to one as $n \to\infty$,
the algorithm reconstructs the labels with accuracy $\tilde p_T(a,b) -
\eps$.
\end{theorem}

Combining Theorem~\ref{teo:alg-gets-ptilde} with Theorems~\ref
{teo:mns:13} and~\ref{teo:robust-tree-intro}
proves Theorem~\ref{teo:alg-is-optimal}.
We remark that
Theorem~\ref{teo:alg-gets-ptilde} easily extends to other versions of
the block model (i.e., models with
more clusters or unbalanced classes);
however, Theorem~\ref{teo:robust-tree-intro} does not. In particular,
Theorem~\ref{teo:alg-is-optimal}
may not hold for general block models. In fact, one fascinating
conjecture of~\cite{Z2011}
says that for general block models, computational hardness enters the
picture (whereas it does
not play any role in our current work).

\subsection{Algorithm outline}

Before getting into the technical details, let us give an outline of our
algorithm: for every node $u$, we remove a neighborhood (whose radius
$r$ is
slowly increasing with $n$) of $u$ from the graph $G$. We then run a
black-box community-detection algorithm on what remains of $G$.
This is guaranteed to produce
some communities which are correlated with the true ones, but they may not
be optimally accurate. Then we return the neighborhood of $u$ to $G$, and
we consider the inferred communities on the boundary of that neighborhood.
Now, the neighborhood of $u$ is like a tree, and the true labels on its
boundary are distributed like $\sigma_{L_r(u)}$. The inferred labels
on the boundary are hence distributed like $\tau_{L_r(u)}$ for some
$0 \le\delta< \frac{1}2$, and so we
can guess the label of $u$ from them using robust tree reconstruction.
(In the previous sentence, we are implicitly claiming that the errors made
by the black-box algorithm are independent of the neighborhood of $u$.
This is because the edges in the neighborhood of $u$ are independent of
the edges in the rest of the graph, a fact that we will justify more carefully
later.)
Since robust tree reconstruction succeeds with probability $p_T$ regardless
of $\delta$, our algorithm attains this optimal accuracy even if the
black-box algorithm does not.

To see the connection between our algorithm and belief propagation,
note that finding the optimal estimator for the tree reconstruction
problem requires computing $\Pr(\sigma_u \mid\tau_{L_r(u)})$. On a tree,
the standard algorithm for solving this is exactly belief propagation. In
other words, our algorithm consists of multiple local applications of
belief propagation. Although we believe that a single global run of belief
propagation would attain the same performance, these local instances are
easier to analyze.

Finally, a word about notation. Throughout this article, we will use the
letters $C$ and $c$ to denote positive constants whose value may change
from line
to line. We will also write statements like ``for all $k \ge K(\theta,
\delta) \dots$''
as abbreviations for statements like ``for every $\theta$ and $\delta
$ there exists
$K$ such that for all $k \ge K \dots.$''

\section{Robust reconstruction on regular trees}
\label{sec:d-ary}

Our main effort is devoted to proving Theorem~\ref{teo:robust-tree-intro}.
Since the proof is quite involved,
we begin with a somewhat easier case of regular trees which already
contains the main ideas of the proof.
The adaptation to the case of Poisson random trees will be carried in
Section~\ref{sec:gw}.

First, we need to define the reconstruction and robust reconstruction
probabilities for regular
trees. Their definitions are analogous to Definitions~\ref{def:tree-accuracy}
and~\ref{def:robust-tree-accuracy}.

%
\begin{definition}
Let $\sigma$ be distributed according to the broadcast process
with parameter $\eta$ on an infinite $d$-ary tree. Let
$\tau$ be distributed according to the noisy broadcast process
with parameters $\eta$ and $\delta$ on the same tree.
We define
\begin{eqnarray*}
p_\mathrm{reg}(d, \eta) &=& \frac{1}2 + \lim_{k\to\infty}
\E\biggl\llvert\Pr(\sigma_\rho=+\mid\sigma_{L_k(\rho)}) -
\frac{1}2 \biggr\rrvert,
\\
\tilde p_\mathrm{reg}(d, \eta) &=& \frac{1}2 + \liminf
_{\delta\to1/2} \liminf_{k\to\infty} \E\biggl\llvert\Pr(
\sigma_\rho=+\mid\tau_{L_k(\rho)}) - \frac{1}2 \biggr
\rrvert.
\end{eqnarray*}
\end{definition}

%
\begin{theorem} \label{teo:robust_tree_regular}
Consider the broadcast process on the infinite $d$-ary tree where if $u
\in L_1(v)$ then
$\Pr(\sigma_u = \sigma_v) = \frac{1}{2}(1+\theta)$ (equivalently
$\E[\sigma_u \sigma_v] = \theta$).
There exists a constant $C$ such that if
$d \theta^2 > C$ then
\[
\tilde p_\mathrm{reg}(d, \eta) = p_\mathrm{reg}(d, \eta).
\]
\end{theorem}

\subsection{Magnetization}
Define
\begin{eqnarray*}
X_{u,k} &=& \Pr(\sigma_u=+ \mid\sigma_{L_k(u)}) -
\Pr(\sigma_u=- \mid\sigma_{L_k(u)}),
\\
x_k &=& \E(X_{u,k} \mid\sigma_u = +).
\end{eqnarray*}
Here, we say that $X_{u,k}$ is the \emph{magnetization} of $u$
given $\sigma_{L_k(u)}$.
Note that by the homogeneity of the tree, the definition
of $x_k$ is independent of $u$. A simple application
of Bayes' rule (see Lemma 1 of~\cite{BCMR06}) shows that $(1 + \E
\llvert X_{\rho,k} \rrvert )/2$ is the
probability of estimating $\sigma_\rho$ correctly given $\sigma
_{L_k(\rho)}$.

We may also define the noisy magnetization $Y$:
%
%
\begin{eqnarray} \label{eq:Y-def}
Y_{u,k} &=& \Pr(\sigma_u=+ \mid\tau_{L_k(u)}) -
\Pr(\sigma_u=- \mid\tau_{L_k(u)}),
\nonumber\\[-8pt]\\[-8pt]\nonumber
y_k &=& \E(Y_{u,k} \mid\sigma_u = +).
\nonumber
\end{eqnarray}
As above, $(1 + \E\llvert Y_{\rho,k}\rrvert )/2$ is the
probability of estimating $\sigma_\rho$ correctly given $\tau
_{L_k(\rho)}$.
In particular, the analogue of Theorem~\ref{teo:robust-tree-intro} for
$d$-ary trees
may be written as follows.
%
%
\begin{theorem}\label{teo:d-ary}
There exists a constant $C$ such that if $\theta^2 d > C$ and $\delta
< \frac{1}2$ then
\[
\lim_{k\to\infty} \E\llvert X_{\rho,k} \rrvert= \lim
_{k\to\infty} \E\llvert Y_{\rho,k}\rrvert.
\]
\end{theorem}

Our main method for proving Theorem~\ref{teo:d-ary} (and also
Theorem~\ref{teo:robust-tree-intro}) is by studying certain
recursions. Indeed,
Bayes' rule implies the following recurrence for $X$ (see, e.g.,~\cite
{Sly11}):
%
%
\begin{eqnarray}
\label{eq:mag-recurrence} X_{u,k} &=& \frac{\prod_{i \in\C(u)} (1 +
\theta X_{ui,k-1}) - \prod_{i \in
\C(u)} (1 - \theta X_{ui,k-1})}{
\prod_{i \in\C(u)} (1 + \theta X_{ui,k-1}) + \prod_{i \in\C(u)}
(1 - \theta X_{ui,k-1})}.
\end{eqnarray}
The same reasoning that gives~\eqref{eq:mag-recurrence} also shows
that~\eqref{eq:mag-recurrence} also holds when every
instance of $X$ is replaced by $Y$. Since our entire analysis is based on
the recurrence~\eqref{eq:mag-recurrence}, the only meaningful (for us)
difference
between $X$ and $Y$ is that their initial conditions are different:
$X_{u,0} = \pm1$ while $Y_{u,0} = \pm(1-2\delta)$. In fact, we will
see later that Theorem~\ref{teo:d-ary} also holds for some more general
estimators $Y$ satisfying~\eqref{eq:mag-recurrence}.

\subsection{The simple majority method}\label{sec:majority}
Our first step in proving Theorem~\ref{teo:d-ary} is to show that
when $\theta^2 d$ is large, then both the exact reconstruction and the
noisy reconstruction do quite well. While it is possible to do so by
studying the recursion~\eqref{eq:mag-recurrence}, such an analysis is
actually quite delicate. Instead, we will show this by studying a completely
different estimator: the one which is equal to the most common label
among $\sigma_{L_k(\rho)}$. This estimator is easy to analyze, and it performs
quite well; since the estimator based on the sign of $X_{\rho,k}$ is optimal,
it performs even better. The study of the simple majority estimator is
quite old,
having essentially appeared in the paper of Kesten and
Stigum~\cite{KestenStigum66}; however, we include most of the details
for the
sake of completeness.

Suppose $d \theta^2 > 1$.
Define $S_{u,k} = \sum_{v \in L_k(u)} \sigma_v$
and set $\widetilde S_{u,k} = \sum_{v \in L_k(u)} \tau_v$. We will
attempt to
estimate $\sigma_\rho$ by $\sgn(S_{\rho,k})$
or $\sgn(\widetilde S_{\rho,k})$; when $\theta^2 d$ is large
enough, these estimators turn out to perform quite well.
We will show this by calculating the first two moments of $S_{u,k}$
and $\widetilde S_{u,k}$; we write $\E^+$ and $\Var^+$ for the conditional
expectation and conditional variance given $\sigma_\rho= +$.
The first moments are trivial, and we omit the proof.

%
\begin{lemma}\label{lem:majority-first-moment}
\begin{eqnarray*}
\E^+ S_{\rho,k} &=& \theta^k d^k,
\\
\E^+ \widetilde S_{\rho,k} &=& (1-2\delta) \theta^k
d^k.
\end{eqnarray*}
\end{lemma}

The second moment calculation uses the recursive structure of the tree.
The argument is not new, but we include it for completeness.
%
%
\begin{lemma}\label{lem:majority-second-moment}
\begin{eqnarray*}
\Var^+ S_{\rho,k} &=& 4 \eta(1-\eta) d^{k} \frac{(\theta^2 d)^k -
1}{\theta^2 d - 1},
\\
\Var^+ \widetilde S_{\rho,k} &=& 4d^k \delta(1-\delta) + 4 (1-2
\delta)^2 \eta(1-\eta) d^{k} \frac{(\theta^2 d)^k - 1}{\theta^2 d - 1}.
\end{eqnarray*}
\end{lemma}

\begin{pf}
We decompose the variance of $S_k$ by conditioning on the first level
of the tree:
%
%
\begin{equation}
\label{eq:variance-decomp} \Var^+ S_{\rho,k} = \E\Var^+(S_{\rho,k} \mid
\sigma_1, \dots, \sigma_d) + \Var^+ \E(S_{\rho,k}
\mid\sigma_1, \dots, \sigma_d).
\end{equation}
Now, $S_{\rho,k} = \sum_{u \in L_1} S_{u,k-1}$, and $S_{u,k-1}$
are i.i.d. under $\Pr^+$. Thus, the first term of~\eqref{eq:variance-decomp}
decomposes into a sum of variances:
\[
\E\Var^+(S_{\rho,k} \mid\sigma_1, \dots,
\sigma_d) = \sum_{u \in L_1} \E
\Var^+(S_{u,k-1} \mid\sigma_u) = d \Var^+(S_{\rho,k-1}).
\]
For the second term of~\eqref{eq:variance-decomp}, note that (by
Lemma~\ref{lem:majority-first-moment}),
$\E(S_{u,k-1} \mid\sigma_u)$ is $(\theta d)^{k-1}$ with probability
$1-\eta$ and $-(\theta d)^{k-1}$ otherwise. Since
$\E(S_{u,k-1} \mid\sigma_u)$ are independent as $u$ varies, we have
\[
\Var^+ \E(S_{\rho,k} \mid\sigma_1, \dots,
\sigma_d) = 4 d \eta(1-\eta) (\theta d)^{2k-2}.
\]
Plugging this back into~\eqref{eq:variance-decomp}, we get the recursion
\[
\Var^+ S_{\rho,k} = d \Var^+ S_{\rho,k-1} + 4 d \eta(1-\eta) (\theta
d)^{2k-2}.
\]
Since $\Var^+ S_{\rho,0} = 0$,
we solve this recursion to obtain
%
%
\begin{eqnarray}
\Var^+ S_{\rho,k} \label{eq:var-S} &=& d\sum_{\ell=1}^k
4 \eta(1-\eta) (\theta d)^{2\ell-2} d^{k-\ell}
\nonumber\\[-8pt]\\[-8pt]\nonumber
&=& 4 \eta(1-\eta) d^{k} \sum_{\ell=0}^{k-1}
\bigl(\theta^2 d \bigr)^{\ell}
= 4 \eta(1-\eta) d^{k} \frac{(\theta^2 d)^k - 1}{\theta^2 d - 1}.
\end{eqnarray}

To\vspace*{1pt} compute $\Var^+ \widetilde S_{\rho,k}$, we condition on $S_{\rho,k}$:
conditioned on $S_{\rho,k}$, $\widetilde S_{\rho,k}$ is a sum of $d^k$
i.i.d. terms,
of which $(d^k + S_{\rho,k}) / 2$ have mean $1-2\delta$,
$(d^k - S_{\rho,k})/2$ have mean $2\delta-1$, and
all have variance $4\delta(1-\delta)$. Hence,
$\E(\widetilde S_k \mid S_k) = (1-2\delta) S_k$ and
$\Var(\widetilde S_k \mid S_k) = 4 d^k \delta(1-\delta)$. By the decomposition
of variance,
\begin{eqnarray*}
\Var^+(\widetilde S_k) &=& \E^+ \bigl(4 d^k \delta(1-\delta)
\bigr) + \Var^+ \bigl((1-2\delta) S_k \bigr)
\\
&=& 4 d^k \delta(1-\delta) + 4 (1-2\delta)^2 \eta(1-\eta)
d^{k} \frac{(\theta^2 d)^k - 1}{\theta^2 d - 1},
\end{eqnarray*}
where the last equality follows from~\eqref{eq:var-S}
and the fact that $\Var(aX) = a^2 \Var(X)$.
\end{pf}

Taking $k \to\infty$ in Lemmas~\ref{lem:majority-first-moment}
and~\ref{lem:majority-second-moment}, we see that
if $\theta^2 d > 1$ then
\[
\left.%
\begin{array} {l}
\displaystyle\frac{\Var^+ S_k}{(\E^+ S_k)^2}
\\[12pt]
\displaystyle\frac{\Var^+ \widetilde S_k} {(\E^+ \widetilde S_k)^2}
\end{array}
\right\} \stackrel{k\to\infty} {\to}
\frac{4\eta(1-\eta)}{\theta^2 d}.
\]
By Chebyshev's inequality, %
\[
\liminf_{k \to\infty} \Pr^+(S_k > 0) \ge1 -
\frac{4\eta(1-\eta)}{\theta^2 d}.
\]
In other words, the estimators
$\sgn(S_k)$ and $\sgn(\widetilde S_k)$ succeed with probability at least
$1 - \frac{4\eta(1-\eta)}{\theta^2 d^2}$ as $k \to\infty$.
Now, $\sgn(Y_{\rho,k})$ is the optimal estimator of $\sigma_\rho$
given~$\tau_{L_k}$, and
its success probability is exactly $(1 + \E\llvert Y_{\rho,k}\rrvert
)/2$. Hence,\vspace*{2pt}
this quantity must be
larger than the success probability of $\sgn(\widetilde S_k)$
[and similarly for $X$ and $\sgn(S_k)$]. Putting this together,
we arrive at the following estimates: if $\theta^2 d > 1$ and $k \ge
K(\delta)$ then
%
%
\begin{eqnarray}
\E\llvert X_{\rho,k} \rrvert&\ge&1 - \frac{10 \eta(1-\eta)}{\theta^2 d},
\label{eq:large-abs-mag-X}
\\
\E\llvert Y_{\rho,k}\rrvert&\ge&1 - \frac{10 \eta(1-\eta)}{\theta^2 d}.
\label{eq:large-abs-mag-Y}
\end{eqnarray}
Now, given that $\sigma_\rho= +$, the optimal estimator makes a mistake
whenever \mbox{$X_{\rho,k} < 0$}; hence,
$\Pr^+(X_{\rho,k} < 0) \le(1 - \E\llvert X_{\rho,k} \rrvert )/2$.
Since $X_{u,k}
\ge-1$, this implies
\[
\E^+ X_{\rho,k} \ge\E^+ \llvert X_{\rho,k} \rrvert- 2
\Pr^+(X_{\rho,k} < 0) \ge1 - \frac{C \eta(1-\eta)}{\theta^2 d}.
\]
We will use this fact repeatedly, so let us summarize in a lemma.
%
%
\begin{lemma}\label{lem:large-expected-mag}
There is a constant $C$ such that
if $\theta^2 d > 1$ and $k \ge K(\delta)$ then
\begin{eqnarray*}
\E^+ X_{\rho,k} &\ge&1 - \frac{C \eta(1-\eta)}{\theta^2 d},
\\
\E^+ Y_{\rho,k} &\ge&1 - \frac{C \eta(1-\eta)}{\theta^2 d}.
\end{eqnarray*}
\end{lemma}

By Markov's inequality, we find that $X_{u,k}$ is large with high probability:


%
\begin{lemma}\label{lem:large-mag}
There is a constant $C$ such that for all $k \ge K(\delta)$ and all $t
> 0$
\begin{eqnarray*}
\Pr\biggl( X_{u,k} \ge1 - t\frac{\eta}{\theta^2 d} \Bigm|
\sigma_u = + \biggr) &\ge&1 - C t^{-1},
\\
\Pr\biggl(Y_{u,k} \ge1 - t\frac{\eta}{\theta^2 d} \Bigm|\sigma
_u = + \biggr) &\ge&1 - C t^{-1}.
\end{eqnarray*}
\end{lemma}

As we will see, Lemma~\ref{lem:large-expected-mag}
and the recursion~\eqref{eq:mag-recurrence} are really the only properties
of $Y$ that we will use. Hence, from now on $Y_{u,k}$ need not be defined
by~\eqref{eq:Y-def}. Rather, we will make the following assumptions on
$Y_{u,k}$.

%
\begin{assumption}\label{ass:Y}
There is a $K = K(\delta)$ such that for all $k \ge K$, the following
hold:
\begin{longlist}[3.]
\item[1.]
$
Y_{u,k+1} = \frac
{\prod_{i \in\C(u)} (1 + \theta Y_{ui,k}) - \prod_{i \in\C(u)} (1
- \theta Y_{ui,k})}{
\prod_{i \in\C(u)} (1 + \theta Y_{ui,k}) + \prod_{i \in\C(u)} (1
- \theta Y_{ui,k})}
$.

\item[2.] The distribution of $Y_{u,k}$ given $\sigma_u = +$ is equal
to the distribution of $-Y_{u,k}$ given $\sigma_u = -$.

\item[3.] $\E^+ Y_{\rho,k} \ge1 - \frac{C \eta(1-\eta)}{\theta^2 d}
$
for some constant $C$.
\end{longlist}
\end{assumption}

We will prove Theorem~\ref{teo:d-ary} under Assumption~\ref{ass:Y}.
Note that part 2 above immediately implies
\[
\E(Y_{ui,k} \mid\sigma_u = +) = \theta\E(Y_{ui,k}
\mid\sigma_{ui} = +).
\]
Also, part 3 implies that Lemma~\ref{lem:large-mag} holds for $Y$.

\subsection{The recursion for small \texorpdfstring{$\theta$}{$theta$}} \label{sec:small-theta}

Our proof of Theorem~\ref{teo:d-ary} proceeds in two cases, with two different
analyses. In the first case, we suppose that $\theta$ is small (i.e., smaller
than a fixed, small constant). In this case, we proceed by Taylor-expanding
the recursion~\eqref{eq:mag-recurrence} in $\theta$.
For the rest of this section, we will assume that $X$ and $Y$ satisfy
parts 1 and 2 of Assumption~\ref{ass:Y}, and that $x_k, y_k \ge5/6$
for $k \ge K(\delta)$.
This restriction will allow us to reuse most of the argument in
the Galton--Watson case (where part 3 of Assumption~\ref{ass:Y} fails
to hold, but we nevertheless have $x_k, y_k \ge5/6$).

%
\begin{proposition}\label{prop:large-d-recursion}
There are absolute constants $C$ and $\theta^* > 0$ such that
if $d \theta^2 \ge C$ and $\theta\le\theta^*$ then for all
$k \ge K(\theta, d, \delta)$,
\[
\E(X_{\rho,k+1} - Y_{\rho,k+1})^2 \le\tfrac{1}2
\E(X_{\rho,k} - Y_{\rho,k})^2.
\]
\end{proposition}

Note\vspace*{1pt} that Proposition~\ref{prop:large-d-recursion} immediately implies
that if $d \theta^2 \ge C$ and $\theta\le\theta^*$ then
$\E(X_{\rho,k} - Y_{\rho,k})^2 \to0$ as $k \to\infty$, which implies
Theorem~\ref{teo:d-ary} in the case that $\theta\le\theta^*$.

In proving Proposition~\ref{prop:large-d-recursion}, the first step
is to replace the right-hand side of~\eqref{eq:mag-recurrence} with something
easier to work with; in particular, we would like to have something
without $X$ in the denominator. For this, we note that
\[
\frac{a - b}{a + b} = \frac{1 - b/a}{1 + b/a} = \frac{2}{1 + b/a} - 1.
\]
Hence, if $a = \prod_i (1 + \theta X_{ui,k})$, $b = \prod_i (1 -
\theta X_{ui,k})$,
and $a'$ and $b'$ are the same quantities with $Y$ replacing $X$, then
%
%
\begin{equation}
\label{eq:rearrange-recursion} \llvert X_{u,k+1} - Y_{u,k+1}\rrvert=
\biggl
\llvert\frac{a-b}{a+b} - \frac{a' - b'}{a'
+ b'} \biggr\rrvert= 2 \biggl\llvert
\frac{1}{1 + b/a} - \frac{1}{1 + b'/a'} \biggr\rrvert.
\end{equation}
Using Taylor's theorem, the right-hand side can be bounded in terms of
$\llvert (b/a)^p - (b'/a')^p\rrvert $ for some $0 < p < 1$ of our choice.

%
\begin{lemma}\label{lem:x^p}
For any $0 < p < 1$ and any $x, y \ge0$,
\[
\biggl\llvert\frac{1}{1+x} - \frac{1}{1+y} \biggr\rrvert\le
\frac{1}p \bigl\llvert x^p - y^p\bigr\rrvert.
\]
\end{lemma}

\begin{pf}
Let $f(x) = \frac{1}{1+x}$ and $g(x) = x^p$. By the fundamental
theorem of calculus, the proof\vspace*{1pt} would follow from the inequality
$\llvert f'(x)\rrvert \le p^{-1} g'(x)$. Now, $\llvert f'(x)\rrvert
= \frac{1}{(1+x)^2}$ and
$g'(x) = p x^{p-1}$. When $x \ge1$, we have
$\llvert f'(x)\rrvert \le x^{-2} \le x^{p-1}$, while if $x \le1$ then
$\llvert f'(x)\rrvert \le1 \le x^{p-1}$.
\end{pf}

As an immediate consequence of Lemma~\ref{lem:x^p} (for $p = 1/4$)
and~\eqref{eq:rearrange-recursion},
%
%
\begin{equation}
\label{eq:get-rid-of-denom} \llvert X_{u,k+1} - Y_{u,k+1}\rrvert\le8
\biggl
\llvert\biggl( \prod_i \frac{1 - \theta X_{ui,k}}{1 + \theta X_{ui,k}}
\biggr)^{1/4} - \biggl(\prod_i
\frac{1 - \theta Y_{ui,k}}{1 + \theta
Y_{ui,k}} \biggr)^{1/4} \biggr\rrvert.
\end{equation}

Next, we present a general bound on the second moment of differences of
products. Of course, we have in mind the example
$A_i = (\frac{1-\theta X_{ui,k}}{1 + \theta X_{ui,k}})^{1/4}$ and
similarly for
$B_i$ and $Y_i$.

%
\begin{lemma}\label{lem:expansion-of-square}
Let $(A_1, B_1), \dots, (A_d, B_d)$ be i.i.d. copies of
$(A, B)$. Then
\[
\E\Biggl(\prod_{i=1}^d A_i
- \prod_{i=1}^d B_i
\Biggr)^2 \le d m^{d-1} \bigl(\E A^2 - \E
B^2 \bigr)^2 + 2 dm^{d-1} \E(A -
B)^2,
\]
where $m = \max\{\E A^2, \E B^2\}$.
\end{lemma}

\begin{pf}
Let $\varepsilon= \E(A_i - B_i)^2$, so that
$\E A_i B_i = \frac{1}{2} (\E A_i^2 + \E B_i^2 - \varepsilon)$. Then
%
%
\begin{eqnarray}\label{eq:expansion-of-square}
\E\Biggl(\prod_{i=1}^d A_i
- \prod_{i=1}^d B_i
\Biggr)^2 &=& \E\prod_{i=1}^d
A_i^2 + \E\prod_{i=1}^d
B_i^2 - 2 \E\prod_{i=1}^d
A_i B_i
\nonumber
\\
&=& \bigl(\E A^2 \bigr)^d + \bigl(\E B^2
\bigr)^d - 2 \prod_{i=1}^d
\frac{\E A_i^2 + \E B_i^2 - \varepsilon}{2}
\\
&=& \bigl(\E A^2 \bigr)^d + \bigl(\E B^2
\bigr)^d - 2 \biggl(\frac{\E A^2 + \E B^2 - \varepsilon}{2} \biggr)^d.\nonumber
\end{eqnarray}
By a second-order Taylor expansion, any twice differentiable $f$
satisfies $f(x) + f(y) \le2 f((x+y)/2) + \frac{1}4 (x-y)^2 \max_z f''(z)$,
where the maximum ranges over $z$ between $x$ and $y$. Applying this for
$f(x) = x^d$ yields
\[
\bigl(\E A^2 \bigr)^d + \bigl(\E B^2
\bigr)^d \le d^2 m^{d-2} \bigl(\E A^2
- \E B^2 \bigr)^2 + 2 \biggl(\frac{\E A^2 + \E B^2}{2}
\biggr)^d.
\]
Hence,
\begin{eqnarray*}
\eqref{eq:expansion-of-square} &\le& d^2 m^{d-2} \bigl(\E
A^2 - \E B^2 \bigr)^2 + 2 \biggl(
\frac{\E A^2 + \E B^2}{2} \biggr)^d - 2 \biggl(\frac{\E A^2 + \E B^2 -
\varepsilon}{2}
\biggr)^d
\\
&\le& d^2 m^{d-2} \bigl(\E A^2 - \E
B^2 \bigr)^2 + 2 d m^{d-1} \varepsilon,
\end{eqnarray*}
where the second inequality follows from a first-order Taylor
expansion of the function $f(x) = x^d$ around $x = (\E A^2 + \E B^2)/2$.
\end{pf}

As we\vspace*{2pt} said before,
we will apply Lemma~\ref{lem:expansion-of-square} with
$A_i = \smash{ (\frac{1 - \theta X_{ui,k}}{1 + \theta
X_{ui,k}} )^{1/4}}$ and
$B_i = \smash{ (\frac{1 - \theta Y_{ui,k}}{1 + \theta
Y_{ui,k}} )^{1/4}}$.
To make the lemma useful, we will need to bound
$\E A_i^2$, $\E B_i^2$, and their difference.
First, we will bound $\E A_i^2$ and $\E B_i^2$. In other words,
we will bound
\[
\E\sqrt{\frac{1 - \theta X_{ui,k}}{1 + \theta X_{ui,k}}}
\]
and the same expression with $Y$ instead of $X$.

%
\begin{lemma}\label{lem:exp-taylor-quotient}
There is a constant $\theta^* > 0$ such that if
$x_k, y_k \ge5/6$ then
\begin{eqnarray*}
\E\bigl(A_i^2 \mid\sigma_u = + \bigr) &
\le&1 - \frac{\theta^2 x_k}{4},
\\
\E\bigl(B_i^2 \mid\sigma_u = + \bigr) &
\le&1 - \frac{\theta^2 y_k}{4}.
\end{eqnarray*}
\end{lemma}

\begin{pf}
First, note that for sufficiently small $x$,
\begin{eqnarray*}
(1 + x) \bigl(1 - x + \tfrac{5}{8} x^2 \bigr)^2
&=& (1+x) \bigl(1 - 2x + \tfrac{18}{8} x^2 + O
\bigl(x^3 \bigr) \biggr)
\\
&=& 1 - x + \tfrac{1}{4} x^2
+ O \bigl(x^3 \bigr) \geq1-x,
\end{eqnarray*}
which may be rearranged to read
\[
\sqrt{\frac{1-x}{1+x}} \le1 - x + \frac{5}{8} x^2.
\]
Now, if $\theta^*$ is sufficiently small then we may apply this with
$x = \theta X_{ui,k}$, obtaining
\[
\E\bigl(A_i^2 \mid\sigma_u = + \bigr)
\le1 - \E(\theta X_{ui,k} \mid\sigma_u = +) +
\tfrac{5}8 \E\bigl(\theta^2 X_{ui,k}^2
\mid\sigma_u = + \bigr).
\]
Recalling the assumption that $x_k \ge5/6$, we have
\begin{eqnarray*}
1 - \E(\theta X_{ui,k} \mid\sigma_u = +) +
\frac{5}8 \E\bigl(\theta^2 X_{ui,k}^2
\mid\sigma_u = + \bigr)
&\le& 1 - \theta^2 x_k +
\frac{3\theta^2}{4} x_k
\\
&=&  1 - \frac{\theta ^2}{4} x_k.
\end{eqnarray*}
The same argument applies to $B_i$, but using $Y_i$ instead of $X_i$.
\end{pf}

\subsection{The $\mathbb{E}A^2-\mathbb{E}B^2$ term}

In this section, we will bound the $\llvert \E A^2 - \E B^2\rrvert $
term in
Lemma~\ref{lem:expansion-of-square}, bearing in mind that the
bound has to be at most of order $\theta^4$ in order for
$d^2 (\E A^2 - \E B^2)^2$ to be a function of $d \theta^2$.
Note that the distribution of $A_i$ conditioned on $\sigma_v = +$ is
equal to the distribution of $1/A_i$ conditioned on $\sigma_v = -$.
Hence,
%
%
\begin{eqnarray}
\label{eq:conditioning-A} \E\bigl(A_i^2 \mid\sigma_u
= + \bigr) &=& (1-\eta) \E\bigl(A_i^2 \mid
\sigma_{ui} = + \bigr) + \eta\E\bigl(A_i^2
\mid\sigma_{ui} = - \bigr)
\nonumber\\[-8pt]\\[-8pt]\nonumber
&=& \E\bigl((1-\eta) A_i^2 + \eta A_i^{-2}
\mid\sigma_{ui} = + \bigr).
\end{eqnarray}
Now,
%
%
\begin{eqnarray}
\label{eq:A-conditioned} (1-\eta) A_i^2 + \eta
A_i^{-2} &=& (1-\eta) \biggl(\frac{1-\theta X_{ui,k}}{1+\theta X_{ui,k}}
\biggr)^{1/2} + \eta\biggl(\frac{1+\theta X_{ui,k}}{1-\theta X_{ui,k}}
\biggr)^{1/2}
\nonumber
\\
&=& \frac{(1-\eta)(1-\theta X_{ui,k}) + \eta(1+\theta X_{ui,k})}{
\sqrt{(1+\theta X_{ui,k})(1-\theta X_{ui,k})}}
\\
&=& \frac{1-\theta^2 X_{ui,k}}{
\sqrt{1 - \theta^2 X_{ui,k}^2}}\nonumber
\end{eqnarray}
(recalling in the last line that $\theta= 1 - 2\eta$).

%
\begin{lemma}\label{lem:diff-bound-on-A}
There is a $\theta^* > 0$ such that
if $\theta< \theta^*$ then
\[
\biggl\llvert\frac{d }{d x} \frac{1-\theta^2 x}{\sqrt{1-\theta^2
x^2}} \biggr\rrvert\le3
\theta^2
\]
for all $x \in[-1, 1]$.
\end{lemma}

\begin{pf}
By a direct computation,
\[
\frac{d }{d x} \frac{1-\theta^2 x}{\sqrt{1-\theta^2 x^2}} = \frac{\theta
^2 x (1-\theta^2 x^2)^{-1/2} (1-\theta^2 x) - \theta
^2\sqrt{1-\theta^2 x^2}}{1-\theta^2 x^2}.
\]
Since $\llvert x\rrvert \le1$, we have
\begin{eqnarray*}
\biggl\llvert\frac{d }{d x} \frac{1-\theta^2 x}{\sqrt{1-\theta^2
x^2}} \biggr\rrvert&\le&
\frac{\theta^2(1-\theta^2)^{-1/2}(1+\theta^2) + \theta^2}{
1-\theta^2}
\\
&=& \theta^2 \frac{(1 - \theta^2)^{-1/2} (1 + \theta^2)
+ 1}{1 - \theta^2}.
\end{eqnarray*}
The result follows because $1 - \theta^2$ and $1 + \theta^2$ can be made
arbitrarily close to 1 by taking $\theta^*$ small enough.
\end{pf}

Now we apply~\eqref{eq:A-conditioned} with Lemma~\ref{lem:diff-bound-on-A}
to obtain the promised bound on \mbox{$\E A_i^2 - \E B_i^2$}.

%
\begin{lemma}\label{lem:A^2-B^2}
There is a $\theta^* > 0$ such that for all
$\theta< \theta^*$,
\[
\E\bigl(A_i^2 - B_i^2 \mid
\sigma_u = + \bigr) \le3 \theta^2 \sqrt{\E
\bigl((X_{ui,k} - Y_{ui,k})^2 \mid
\sigma_{u} = + \bigr)}.
\]
\end{lemma}

\begin{pf}
By~\eqref{eq:conditioning-A} and~\eqref{eq:A-conditioned} (and analogously
with $A$ replaced by $B$), we have
\[
\E\bigl(A_i^2 - B_i^2 \mid
\sigma_u = + \bigr) = \E\biggl( \frac{1-\theta^2 X_{ui,k}}{\sqrt
{1-\theta^2 X_{ui,k}^2}} -
\frac{1-\theta^2 Y_{ui,k}}{\sqrt{1-\theta^2 Y_{ui,k}^2}} \Bigm|\sigma
_{ui} = + \biggr).
\]
For a general function $f$ we have $\E\llvert f(X) - f(Y)\rrvert
\le\E\llvert X - Y\rrvert \max_x \llvert\frac{d f}{d x} \rrvert$.
Applying this fact
with the
function $f(x) = \frac{1 - \theta^2 x}{\sqrt{1-\theta^2x^2}}$ and
the bound
of Lemma~\ref{lem:diff-bound-on-A},
\begin{eqnarray*}
\E\bigl(A_i^2 - B_i^2 \mid
\sigma_u = + \bigr) &\le&3 \theta^2 \E\bigl(\llvert
X_{ui,k} - Y_{ui,k}\rrvert\mid\sigma_{ui} = +\bigr)
\\
&\le&3 \theta^2 \sqrt{\E\bigl((X_{ui,k} -
Y_{ui,k})^2 \mid\sigma_{ui} = + \bigr)}.
\end{eqnarray*}

Finally, note that
\[
\E\bigl((X_{ui,k} - Y_{ui,k})^2 \mid
\sigma_{ui} = + \bigr) = \E\bigl((X_{ui,k} -
Y_{ui,k})^2 \mid\sigma_{u} = + \bigr).
\]
\upqed
\end{pf}

\subsection{Combining the estimates to complete the proof}
Next, we combine Lemma~\ref{lem:expansion-of-square} with
the estimates provided in Lemmas~\ref{lem:exp-taylor-quotient}
and~\ref{lem:A^2-B^2}.
%
%
\begin{lemma}\label{lem:recursion-ratio}
There is some constant $\theta^* > 0$ such that the following
holds.
Suppose that $X$ and $Y$ satisfy parts 1 and 2 of
Assumption~\ref{ass:Y} and that $x_k, y_k \ge5/6$ for
$k \ge K(\delta)$. If $u$ has $d \ge4$ children and
$\theta\le\theta^*$ then for $k \ge K(\delta)$,
\[
\frac{\E((X_{u,k+1} - Y_{u,k+1})^2 \mid\sigma_u = + )}{
\E((X_{u1,k} - Y_{u1,k})^2 \mid\sigma_{u1} = + )} \le C \bigl(d^2
\theta^4 + d
\theta^2 \bigr) e^{-\sfrac{\theta^2 d}{5}},
\]
for a universal constant $C$.
\end{lemma}

\begin{pf}
Taking the square of~\eqref{eq:get-rid-of-denom} and taking the
expectation on
both sides, we have
\[
\E\bigl((X_{u,k+1} - Y_{u,k+1})^2 \mid
\sigma_u = + \bigr) \le64 \E\Biggl( \Biggl(\prod
_{i=1}^d A_i - \prod
_{i=1}^d B_i \Biggr)^2
\biggm|\sigma_u = + \Biggr).
\]
Conditioned on $\sigma_u$, the pairs $(A_i, B_i)$ are i.i.d. and so
Lemma~\ref{lem:expansion-of-square} implies that
%
%
\begin{eqnarray}
\label{eq:apply-expansion}
&& \E\bigl((X_{u,k+1} - Y_{u,k+1})^2
\mid\sigma_u = + \bigr)
\nonumber\\[-8pt]\\[-8pt]\nonumber
&&\qquad \le64 d^2 m^{d-2} (a - b)^2 + 128 d
m^{d-1} \E\bigl((A_i - B_i)^2
\mid\sigma_u = + \bigr),
\end{eqnarray}
where
\begin{eqnarray*}
a &=& \E\bigl(A_i^2 \mid\sigma_u = +
\bigr),
\\
b &=& \E\bigl(B_i^2 \mid\sigma_u = +
\bigr),
\\
m &=& \max\{a, b\}.
\end{eqnarray*}

Now, if $\theta^*$ is sufficiently small then the function
$x \mapsto(\frac{1-\theta x}{1+\theta x})^{1/4}$ has derivative at most
$\theta$ for
$x \in[-1, 1]$. Hence,
%
%
\begin{eqnarray}\label{eq:(A-B)^2}
\E\bigl((A_i - B_i)^2 \mid
\sigma_u = + \bigr) &\le&\theta^2 \E\bigl((X_{u1,k}
- Y_{u1,k})^2 \mid\sigma_u = + \bigr)
\nonumber\\[-8pt]\\[-8pt]\nonumber
&=& \theta^2 \E\bigl((X_{u1,k} - Y_{u1,k})^2
\mid\sigma_{u1} \bigr)
\end{eqnarray}
provided that $\theta^*$ is sufficiently small. Define
\[
z = \E\bigl((X_{u1,k} - Y_{u1,k})^2 \mid
\sigma_{u1} \bigr) = \E\bigl((X_{u1,k} - Y_{u1,k})^2
\mid\sigma_{u1} = + \bigr).
\]

By Lemma~\ref{lem:exp-taylor-quotient} and the assumption that
$x_k, y_k \ge5/6$, if $\theta^*$ is sufficiently small, then
$m \le1 - \theta^2/5 \le\exp(-\theta^2/5)$. Moreover,
Lemma~\ref{lem:A^2-B^2} implies that $(a - b)^2 \le9 \theta^4 z$.
Plugging these and~\eqref{eq:(A-B)^2} back into~\eqref
{eq:apply-expansion}, we have
\[
\E\bigl((X_{u,k+1} - Y_{u,k+1})^2 \mid
\sigma_u = + \bigr) \le64 \bigl(9 d^2
\theta^4 e^{-\sfrac{\theta^2 (d-2)}{5}} + 2 d \theta^2 e^{-\sfrac{\theta
^2 (d-1)}{5}}
\bigr) z,
\]
which proves the claim.
\end{pf}

\begin{pf*}{Proof of Proposition~\ref{prop:large-d-recursion}}
If $\theta^2 d$ is
sufficiently large, then Lemma~\ref{lem:large-expected-mag} implies
that $x_k, y_k \ge5/6$ for $k \ge K(\delta)$; hence,\vspace*{2pt} the conditions
of Lemma~\ref{lem:recursion-ratio} are satisfied. Finally,
if $d \theta^2$ is large enough then the right-hand side
in Lemma~\ref{lem:recursion-ratio} is at most $\frac{1}2$.
\end{pf*}

\subsection{The recursion for large \texorpdfstring{$\theta$}{$theta$}}
To handle the case in which $\theta$ is not small, we require a different
argument. In this case, we study the derivatives of the recurrence,
obtaining the following result.
%
%
\begin{proposition}\label{prop:large-theta}
For any $0 < \theta^* < 1$, there is some $d^* = d^*(\theta^*)$ such that
for all $\theta\ge\theta^*$, $d \ge d^*$, and $k \ge K(\theta, d,
\delta)$,
\[
\E\sqrt{\llvert X_{\rho,k+1} - Y_{\rho,k+1}\rrvert} \le
\tfrac{1}2 \E\sqrt{\llvert X_{\rho,k} - Y_{\rho,k}\rrvert}.
\]
\end{proposition}

Combined with Proposition~\ref{prop:large-d-recursion}, this proves
Theorem~\ref{teo:d-ary}.
Indeed, to complete the choices of parameters we first take $\theta^*$
to be
the universal constant in Proposition~\ref{prop:large-d-recursion}.
Then let $d^* = d^*(\theta^*)$
be given by Proposition~\ref{prop:large-theta} (note that $d^*$ is
also a universal constant).
Finally, choose $C$ to be the maximum of
$d^*$ and the $C$ from Proposition~\ref{prop:large-d-recursion}.
Now, if $\theta^2 d \ge C$ then either $\theta\le\theta^*$ in which
case Proposition~\ref{prop:large-d-recursion}
applies, or $\theta\ge\theta^*$ in which case $\theta\le1$ implies
that $d \ge C \ge d^*$
and so Proposition~\ref{prop:large-theta} applies. In either case, we
deduce Theorem~\ref{teo:d-ary}.

Let $g: \R^d \to\R$ denote the function
%
%
\begin{equation}
\label{eq:g-def} g(x) = \frac{\prod_{i=1}^d (1 + \theta x_i) - \prod
_{i=1}^d (1 -
\theta x_i)}{
\prod_{i=1}^d (1 + \theta x_i) + \prod_{i=1}^d (1 - \theta x_i)}.
\end{equation}
Then the recurrence~\eqref{eq:mag-recurrence} may be written as
$X_{u,k+1} = g(X_{u1,k}, \dots, X_{ud,k})$.
We will also abbreviate $(X_{u1,k}, \dots, X_{ud,k})$ by
$X_{L_1(u),k}$, so that we may write
$X_{u,k+1} = g(X_{L_1(u),k})$.

Define $g_1(x) = \prod_{i=1}^d (1 + \theta x_i)$ and $g_2(x) = \prod_{i=1}^d
(1 - \theta x_i)$ so that $g$ can be written as
$g = \frac{g_1 - g_2}{g_1 + g_2}$. Since
$\frac{\partial g_1}{\partial x_i} = \theta\frac{g_1}{1 + \theta
x_i}$ and
$\frac{\partial g_2}{\partial x_i} = - \theta\frac{g_2}{1 - \theta
x_i}$, we have
%
%
\begin{eqnarray}\label{eq:diff-formula}
\frac{\partial g}{\partial x_i} &=& \frac{\partial}{\partial x_i} \frac
{g_1 - g_2}{g_1 + g_2}
\nonumber
\\
&=& 2 \frac{g_2 \frac{\partial g_1}{\partial x_i} - g_1 \frac
{\partial g_2}{\partial x_i}}{(g_1 +
g_2)^2}
\\
&=& 4 \theta\frac{g_1 g_2}{(g_1 + g_2)^2 (1 - \theta^2 x_i^2)}.\nonumber
\end{eqnarray}
If $\llvert x_i\rrvert \le1$, then $g_1$ and $g_2$ are both positive, so
$\frac{g_1 g_2}{(g_1 + g_2)^2} \le\frac{g_1 g_2}{g_1^2} = \frac{g_2}{g_1}$;
of course, we also have the symmetric bound
$\frac{g_1 g_2}{(g_1 + g_2)^2} \le\frac{g_1}{g_2}$.
Define
\begin{eqnarray*}
h_i^+(x) &=& 4 \frac{g_2}{(1 - \theta^2 x_i^2 )g_1} = \frac{4}{(1+\theta
x_i)^2} \prod
_{j \ne i} \frac{1-\theta
x_j}{1+\theta x_j},
\\
h_i^-(x) &=& 4 \frac{g_1}{(1 - \theta^2 x_i^2 )g_2} = \frac{4}{(1-\theta
x_i)^2} \prod
_{j \ne i} \frac{1+\theta
x_j}{1-\theta x_j},
\\
h_i(x) &=& \min\bigl\{h_i^+(x), h_i^-(x)
\bigr\}.
\end{eqnarray*}
By~\eqref{eq:diff-formula} and since $\llvert \theta\rrvert \le1$,
%
%
\begin{equation}
\label{eq:diff-bound-2/1} \biggl\llvert\frac{\partial g}{\partial x_i}
\biggr\rrvert\le
h_i(x).
\end{equation}
The point is that if $\sigma_u = +$ then for most $v \in L_1(u)$,
$X_{v,k}$ will be close to 1 and so $h_i^+(X_{L_1(u), k})$ will be small.
On the other hand, if $\sigma_u = -$ then for most $v \in L_1(u)$,
$X_{v,k}$ will be close to $-1$ and so $h_i^-(X_{L_1(u), k})$
will be small.

Note that $h_i^+$ is convex on $[-1, 1]^d$
because it is the tensor product of nonnegative, convex functions.
Hence, for any
$x, y \in[-1, 1]^d$ and any $0 < \lambda< 1$,
\[
\biggl\llvert\frac{\partial g}{\partial x_i} \bigl(\lambda x +
(1-\lambda) y \bigr) \biggr
\rrvert\le h_i^+ \bigl(\lambda x + (1-\lambda) y \bigr) \le\max
\bigl\{h_i^+(x), h_i^+(y) \bigr\}.
\]
Then the mean value theorem implies that
\[
\bigl\llvert g(x) - g(y)\bigr\rrvert\le\sum_i
\llvert x_i - y_i\rrvert\max\bigl\{h_i^+(x),
h_i^+(y) \bigr\}.
\]
Applied for $x = X_{L_1(u),k} = (X_{u1,k}, \dots, X_{ud,k})$ and
$y = Y_{L_1(u),k} = (Y_{u1,k}, \dots,\break  Y_{ud,k})$, this yields
%
%
\begin{eqnarray}\label{eq:mvt-g}
&& \llvert X_{u,k+1} - Y_{u,k+1}\rrvert
\nonumber\\[-8pt]\\[-8pt]\nonumber
&&\qquad \le\sum _i \llvert X_{ui,k} - Y_{ui,k}
\rrvert\max\bigl\{ h_i^+(X_{L_1(u),k}), h_i^+(Y_{L_1(u),k})
\bigr\}.
\end{eqnarray}
Note that the two terms on the right-hand side of~\eqref{eq:mvt-g} are
dependent on one another. Hence, it will be convenient to bound
$h_i^+(X_{L_1(u),k})$ by something that does not depend
on $X_{ui}$. To that end, note that for $\llvert x_i\rrvert \le1$, we have
$1 + \theta x_i \ge1-\theta= 2\eta$, and so
%
%
\begin{equation}
\label{eq:mi-def} h_i^+(x) = \frac{4}{(1 + \theta x_i)^2} \prod
_{j\ne i} \frac{1-\theta
x_j}{1+\theta x_j} \le\frac{1}{\eta^2} \prod
_{j \ne i} \frac{1-\theta x_j}{1+\theta x_j} =: m_i(x).
\end{equation}
Since $m_i(x)$ does not depend on $x_i$, it follows that
$m_i(X_{L_1(u),k})$ is
independent of $X_{ui,k}$ given $\sigma_u$ (and similarly with $Y$
instead of $X$).
Hence,~\eqref{eq:mvt-g} implies that
%
%
\begin{eqnarray}\label{eq:bound-by-m}
&& \E\bigl(\sqrt{\llvert X_{u,k+1} - Y_{u,k+1}\rrvert} \mid
\sigma_u=+ \bigr)\nonumber
\\
&&\qquad \le\sum_i \E\bigl(\sqrt{\llvert
X_{ui,k} - Y_{ui,k}\rrvert}\mid\sigma_u=+ \bigr)
\\
&&\quad\qquad{}\times
\E\bigl(\sqrt{\max\bigl\{m_i(X_{L_1(u),k}),
m_i(Y_{L_1(u),k}) \bigr\}} \mid\sigma_u=+ \bigr).\nonumber
\end{eqnarray}

To prove Proposition~\ref{prop:large-theta}, it therefore suffices to
show that\break $\E(\sqrt{m_i(X_{L_1(u),k})} \mid\sigma_u=+)$ and
$\E(\sqrt{m_i(Y_{L_1(u),k})} \mid\sigma_u=+)$ are both small.
Since $m_i(X_{L_1(u),k})$ is a product of independent (when conditioned
on $\sigma_u$) terms, it is enough to show that each of these terms has
small expectation. The following lemma will help bounding these terms.

%
\begin{lemma}\label{lem:expectation-sqrt}
For any $0 < \theta^* < 1$, there is some $d^* = d^*(\theta^*)$ and some
$\lambda= \lambda(\theta^*) < 1$ such that
for all $\theta\ge\theta^*$, $d \ge d^*$ and $k \ge K(\theta, d,
\delta)$,
\[
\E\biggl(\sqrt{\frac{1-\theta X_{ui,k}}{1+\theta X_{ui,k}}} \Bigm|\sigma
_u=+ \biggr) \le\min
\bigl\{\lambda, 4\eta^{1/4} \bigr\}.
\]
\end{lemma}

The proof of Lemma~\ref{lem:expectation-sqrt} is straightforward but
tedious, and we postpone
it until the \hyperref[app]{Appendix}. Instead, we will now prove Proposition~\ref
{prop:large-theta}.

\begin{pf*}{Proof of Proposition~\ref{prop:large-theta}}
By Lemma~\ref{lem:expectation-sqrt}, and the definition~\eqref{eq:mi-def}
of $m_i$, it follows that
%
%
\begin{eqnarray}\label{eq:mi-bound}
\E\bigl(\sqrt{m_i(X_{ui,k})} \mid
\sigma_u = + \bigr) &\le& \eta^{-1} \min\bigl\{ \lambda,
\eta^{1/4} \bigr\}^{d-1}
\nonumber\\[-8pt]\\[-8pt]\nonumber
&\le& \min\bigl\{\lambda,\eta^{1/4} \bigr\}^{d-5} \le\lambda^{d-5}.
\end{eqnarray}
In particular, if $d^*(\theta^*)$ is sufficiently large then $d
\lambda^{d-5} \le
1/4$ for all $d \ge d^*$.
The same argument applies with $Y$ replacing $X$, and hence
%
%
\begin{equation}
\label{mi-bound-2} \E\bigl(\sqrt{\max\{m_i(X_{L_1(u),k}),
m_i(Y_{L_1(u),k})} \mid\sigma_u = + \bigr) \le
\frac{1}{2d}.
\end{equation}
By~\eqref{eq:bound-by-m}, we have
\[
\E\bigl(\sqrt{\llvert X_{u,k+1} - Y_{u,k+1}\rrvert} \mid
\sigma_u = + \bigr) \le\tfrac{1}{2} \E\bigl(\sqrt{\llvert
X_{u,k} - Y_{u,k}\rrvert} \mid\sigma_u = +
\bigr),
\]
and so we have proved Proposition~\ref{prop:large-theta}.
\end{pf*}

\section{Reconstruction accuracy on Galton--Watson trees}
\label{sec:gw}

In this section, we will adapt the proof of the $d$-ary case
(Theorem~\ref{teo:d-ary}) to the Galton--Watson case (Theorem~\ref
{teo:robust-tree-intro}).
Let $T \subset\N^*$ be a Galton--Watson tree with offspring distribution
$\Pois(d)$. Recall that such a tree may be constructed by taking,
for each $u \in\N^*$, an independent $\Pois(d)$ random variable
$D_u$. Then define $T \subset\N^*$ recursively by starting with
$\varnothing\in T$ and then taking $ui \in T$ for $i \in\N$
if $u \in T$ and $i \le D_u$.

As in Section~\ref{sec:d-ary}, we let $\{\sigma_u: u \in T\}$ be distributed
as the two-state broadcast process on $T$ with parameter $\eta$, and let
$\{\tau_u: u \in T\}$ be the noisy version, with parameter $\delta$.
We recall the magnetization
\begin{eqnarray*}
X_{u,k} &=& \Pr(\sigma_u = + \mid\sigma_{L_k(u)}) -
\Pr(\sigma_u = - \mid\sigma_{L_k(u)}),
\\
x_k &=& \E(X_{u,k} \mid\sigma_u=+).
\end{eqnarray*}
Note that unlike in Section~\ref{sec:d-ary}, $X_{u,k}$ now depends on both
the randomness of the tree
and the randomness of $\sigma$. Hence, $x_k$ now averages over both
the randomness of
the tree and the randomness of $\sigma$.

We recall that $X$ satisfies the recursion~\eqref{eq:mag-recurrence}.
As in Section~\ref{sec:d-ary}, we will let $\{Y_{u,k}\}$ be any collection
of random variables which satisfies the same recursion (for large
enough $k$),
and for which $Y_{u,k}$ is a good estimator of $\sigma_u$ given
$\sigma_{L_k(u)}$.
%
%
\begin{assumption}\label{ass:gw-Y}
There is a $K = K(\delta)$ and a constant $C$ such that for all $k \ge
K$, the following
hold:
\begin{longlist}[2.]
\item[1.]
$
Y_{u,k+1} = \frac
{\prod_{i \in\C(u)} (1 + \theta Y_{ui,k}) - \prod_{i \in\C(u)} (1
- \theta Y_{ui,k})}{
\prod_{i \in\C(u)} (1 + \theta Y_{ui,k}) + \prod_{i \in\C(u)} (1
- \theta Y_{ui,k})}$.

\item[2.] The distribution of $Y_{u,k}$ given $\sigma_u = +$ is equal
to the distribution of $-Y_{u,k}$ given $\sigma_u = -$.

\item[3.] With probability at least $1 - e^{-cd}$ over $T$,
\[
\E(Y_{u,k} \mid\sigma_u = +, T) \ge1 -
\frac{C \eta}{\theta^2
d}.
\]
\end{longlist}
\end{assumption}

Note that Assumption~\ref{ass:gw-Y} is the same as Assumption~\ref{ass:Y}
except for part 3. Indeed, the change in part 3 between Assumption~\ref{ass:Y}
and Assumption~\ref{ass:gw-Y} points to the main change, and biggest challenge,
in extending our previous argument to Galton--Watson trees: unlike for
a regular tree,
there is always some chance that a Galton--Watson tree will go extinct,
or that it
will be thinner and more spindly than expected. In this case, we will
not be able
to reconstruct the broadcast process as well as we might want, even as
$\eta\to0$.

In any case, in order to prove Theorem~\ref{teo:robust-tree-intro}
it suffices to prove that $Y$ satisfies part 3 of Assumption~\ref{ass:gw-Y}
as well as the following theorem.

%
\begin{theorem}\label{teo:gw-main}
Under Assumption~\ref{ass:gw-Y},
there is a universal constant $C$ such that if $\theta^2 d \ge C$ then
$\lim_{k \to\infty} \E\llvert X_{\rho,k} \rrvert = \lim_{k \to\infty}
\E
\llvert Y_{\rho,k}\rrvert $.
\end{theorem}

Recall that $p_T(a,b)$ is equal to $\lim_{k \to\infty} (1 + \E
\llvert X_{\rho,k} \rrvert )/2$ in the
case $d = (a + b)/2$ and $\eta= b/(a + b)$, and that $\tilde p_T(a,b)$ is
equal to $\lim_{k\to\infty} (1 + \E\llvert Y_{\rho,k}\rrvert )/2$ in
the same
case. In particular,
Theorem~\ref{teo:gw-main} immediately implies Theorem~\ref
{teo:robust-tree-intro}.

\subsection{Large expected magnetization}

The first step toward extending
Theorem~\ref{teo:d-ary} to the Galton--Watson case is to show
that the magnetization of each node tends to be large.

%
\begin{proposition}\label{prop:gw-large-expected-mag}
There is a universal constant $c > 0$ such that for all $k \ge K(\theta
, d, \delta)$,
\[
\Pr\biggl(\E(X_{\rho,k} \mid\sigma_\rho= +, T) \ge1 -
\frac
{16\eta}{\theta^2 d} \biggr) \ge1 - e^{-cd}
\]
and similarly for $Y_{\rho,k}$.
Hence, $x_k,y_k \ge1 - \frac{8\eta}{\theta^2 d} - 2 e^{-cd}$.
\end{proposition}

Note that the proposition implies that $Y$ satisfies part $3$ of
Assumption~\ref{ass:gw-Y}.

In the regular case, the proof of Lemma~\ref{lem:large-expected-mag}
was based on the fact that a simple
majority vote at the leaves estimates the root well. Here, we will
follow Evans et~al. \cite{EvKePeSc00} by using a weighted majority vote.
For this, we will need to use the terminology of electrical networks,
in particular
the notion of effective conductance and effective resistance.
An \hyperref[sec1]{Introduction} to these concepts may be found
in~\cite{Lyons13}; the essential properties that we will need are that
conductances add over parallel paths, while resistances add over consecutive
paths.

Put a resistance of $(1-\theta^2) \theta^{-2k}$ on each
edge $e$ in $T$ whose child is in generation $k$ (where $\rho$ is
generation zero).
We write $\cond(k)$ for the
effective conductance between $\rho$ and level $k$ and
$\res(k)$ for $1/\cond(k)$.
Also, attach an additional ``noisy'' node to each node at level $k$, with
resistance $4 \delta(1-\delta) (1-2\delta)^{-2} \theta^{-2k}$; then
let $\mathscr{C}'_{\mathrm{eff}}(k)$ be
the effective conductance between the root and these nodes and let
$\mathscr{R}'_{\mathrm{eff}}(k) = 1/\mathscr{C}'_{\mathrm{eff}}(k)$.
Note that $\cond(k)$ and $\mathscr{C}'_{\mathrm{eff}}(k)$ are random
quantities which
depend on the
Galton--Watson tree.
The importance of $\cond$ and $\mathscr{C}'_{\mathrm{eff}}$ for
estimating $\sigma_\rho$ was
shown by~\cite{EvKePeSc00} (Lemma 5.1).

%
\begin{theorem}\label{teo:weighted-maj}
There exist weights $w(u)$ such that
if $R_k = \sum_{v \in L_k(\rho)} w(v) \sigma_v$ and
$S_k = (1-2\delta)^{-1} \sum_{v \in L_k(\rho)} w(v) \tau_v$ then
\begin{eqnarray*}
\E(R_k \mid\sigma_\rho) &=& \sigma_\rho,
\\
\E(S_k \mid\sigma_\rho) &=& \sigma_\rho,
\\
\Var(R_k \mid\sigma_\rho) &=& \res(k),
\\
\Var(S_k \mid\sigma_\rho) &=& \mathscr{R}'_{\mathrm{eff}}(k).
\end{eqnarray*}
\end{theorem}

We mention that $w(v)$ in Theorem~\ref{teo:weighted-maj}
is proportional to the unit current flow from $\rho$ to $v$; for our work,
however, we only need to know that it exists and that it can be easily computed.

Consider the estimators $\sgn(R_k)$ and $\sgn(S_k)$ for $\sigma_\rho
$. By
Chebyshev's inequality,
\[
\Pr(S_k \le0 \mid\sigma_\rho= +) \le\Var(S_k)
= \res(k) = \frac{1}{\cond(k)}
\]
and similarly $\Pr(R_k \le0 \mid\sigma_\rho= +) \le1/\mathscr
{C}'_{\mathrm{eff}}(k)$.
In particular, if we can show that $\cond(k)$ and $\mathscr
{C}'_{\mathrm{eff}}(k)$ are
large, we
will have shown that $\sgn(S_k)$ and $\sgn(R_k)$ are good estimators
of $\sigma_\rho$. Since $\sgn(X_{k,\rho})$ and $\sgn(Y_{k,\rho})$ are
the optimal estimators of $\sigma_\rho$ given, respectively, $\sigma
_{L_k(\rho)}$
and $\tau_{L_k(\rho)}$, this will prove that $x_k$ and $y_k$ are large.
Note that this is exactly the same method that we used to show that
$x_k$ and $y_k$ are large in the $d$-regular case; the difference here is
that we need to consider a weighted linear estimator instead of
an unweighted one.

%
\begin{lemma}\label{lem:conductance}
There is a universal constant $c > 0$ such that
for all $k \ge K(\theta, d, \delta)$,
\begin{eqnarray*}
\Pr\biggl(\cond(k) \ge\frac{\theta^2 d}{16\eta} \biggr) &\ge& e^{-cd},
\\
\Pr\biggl(\mathscr{C}'_{\mathrm{eff}}(k) \ge\frac{\theta^2
d}{16\eta}
\biggr) &\ge& e^{-cd}.
\end{eqnarray*}
\end{lemma}

\begin{pf}
The proof is by a recursive argument.
Note that $\cond(0) = \infty$ and $\mathscr{C}'_{\mathrm{eff}}(0) =
(4 \delta(1-\delta
))^{-1} (1-2\delta)^{-2} > 0$.
We will write the rest of the proof only for $\cond$, but the same
argument holds with $\mathscr{C}'_{\mathrm{eff}}$ replacing $\cond$
everywhere.
Let $\alpha_{k-1} = \min\{(4\eta)^{-1}, M\}$ where $M$ is the
largest median of $\cond(k-1)$
[in the case of $\cond(0)$, $M$ is any positive value].
Now fix $k$ and let $Z_1, Z_2, \dots$ be independent copies of $\cond(k-1)$.
Then $\Pr(Z_i \ge\alpha_{k-1}) \ge1/2$ for all $i$.

Now, the first $k$ levels of a Galton--Watson tree consist of a root with
$\Pois(d)$ independent subtrees of $k-1$ levels each. For each child
$i$, the conductance
between $i$ and $L_{k-1}(i)$ is distributed like $\theta^{2} Z_i$ (the
factor $\theta^2$
arises because at each level of the tree the conductances are
multiplied by an extra factor
of $\theta^2$). Since the edge between $\rho$ and $i$ has conductance
$\theta^2 (1-\theta^2)^{-1}$,
the conductance between $\rho$ and $L_{k-1}(i)$ is distributed like
\[
\frac{1}{\theta^{-2} Z_i^{-1} + \theta^{-2} (1-\theta^2)} = \frac
{\theta^2 Z_i}{(1-\theta^2) Z_i + 1}.
\]
Summing over the children of $\rho$, we see that $\cond(k)$ has the
same distribution as
\[
\sum_{i=1}^{\Pois(d)} \frac{\theta^2 Z_i}{(1-\theta^2) Z_i + 1} \ge
\theta^2 \sum_{i=1}^{\Pois(d)}
\frac{Z_i}{4\eta Z_i+1}.
\]
Recall that $\Pr(Z_i \ge\alpha_{k-1}) \ge1/2$ and $\alpha_{k-1}
\le(4\eta)^{-1}$. Hence,
$\alpha_{k-1}/(4\eta\alpha_{k-1} + 1) \ge\alpha_{k-1}/2$, and so
\begin{eqnarray*}
\cond(k) &\ge&\theta^2 \sum_{i=1}^{\Pois(d)}
1_{\{Z_i \ge\alpha_{k-1}\}} \frac{\alpha_{k-1}}{4\eta\alpha_{k-1} + 1}
\\
&\ge&\frac{\theta^2}{2} \sum_{i=1}^{\Pois(d)}
1_{\{Z_i \ge\alpha
_{k-1}\}} \alpha_{k-1}
\\
&\ge&\frac{\theta^2\alpha_{k-1}}{2} \Pois(d/2).
\end{eqnarray*}
Now, there is a universal constant $c > 0$ such that $\Pr(\Pois(d/2)
\le d/4) \le e^{-c d}$; hence
%
%
\begin{equation}
\label{eq:large-conductance-1} \Pr\bigl(\cond(k) \le\theta^2 d
\alpha_{k-1}/4 \bigr) \le e^{-cd}.
\end{equation}
In particular, if $d$ is sufficiently large then $e^{-cd} < 1/2$, and hence
every median of $\cond(k)$ is larger than $\theta^2 d \alpha_{k-1}
/4$. In particular,
$\alpha_k \ge\min\{(4\eta)^{-1},\break \theta^2 d \alpha_{k-1} / 4\}$.
Hence, if $\theta^2 d > 4$ and $k$
is sufficiently large then $\alpha_k \ge(4\eta)^{-1}$.
Applying this to~\eqref{eq:large-conductance-1} completes the proof
for $\cond(k)$, and an identical
argument applies to $\mathscr{C}'_{\mathrm{eff}}(k)$.
\end{pf}

Now Proposition~\ref{prop:gw-large-expected-mag} follows directly
from Theorem~\ref{teo:weighted-maj} and Lemma~\ref{lem:conductance}.

\subsection{The small-\texorpdfstring{$\theta$}{$theta$} case}

The proof of Proposition~\ref{prop:large-d-recursion} extends fairly easily
to the Galton--Watson case. The weakening of Lemma~\ref{lem:large-expected-mag}
to Proposition~\ref{prop:gw-large-expected-mag} makes hardly any difference
because the proof of Proposition~\ref{prop:large-d-recursion} only
needed $x_k \ge1/2$.

%
\begin{proposition}\label{prop:large-d-recursion-gw}
Consider the broadcast process on a Poisson Galton--Watson tree. Then
there are absolute constants $C$ and $\theta^* > 0$ such that
if $d \theta^2 \ge C$ and $\theta\le\theta^*$ then for all
$k \ge K(\theta, d, \delta)$,
\[
\E(X_{\rho,k+1} - Y_{\rho,k+1})^2 \le\tfrac{1}2
\E(X_{\rho,k} - Y_{\rho,k})^2.
\]
\end{proposition}

\begin{pf}
Let $D$ be the number of children of $u$, so that $D \sim\Pois(d)$.
If $\theta^2 d$ is sufficiently large then Proposition~\ref
{prop:gw-large-expected-mag}
implies that $x_k, y_k \ge5/6$ and so applying Lemma~\ref{lem:recursion-ratio}
conditioned on $D$ yields
\begin{eqnarray*}
\E\bigl((X_{u,k+1} - Y_{u,k+1})^2 \mid D,
\sigma_u = + \bigr) &\le& C \bigl( D^2
\theta^4 + D \theta^2 \bigr) e^{-\sfrac{\theta^2 D}{5}} z
\\
&\le& C'e^{-\sfrac{\theta^2 D}{10}} z,
\end{eqnarray*}
where $z = \E((X_{u1,k} - Y_{u1,k})^2 \mid\sigma_{u1} = +)$.
Now we integrate out $D$. Since $D \sim\Pois(d)$, its moment
generating function
is $\E e^{t D} = e^{d(e^t - 1)}$. Setting $t = -\theta^2/10$,
we have $e^t \le1 + t/2$ for all $\theta\in[0, 1]$; hence,
\[
\E e^{tD} \le e^{td/2} = e^{-\sfrac{\theta^2 d}{20}}.
\]
That is,
\[
\E\bigl((X_{u,k+1} - Y_{u,k+1})^2 \mid
\sigma_u = + \bigr) \le C z \E e^{-\sfrac{\theta^2 D}{10}} \le C z
e^{-\sfrac{\theta^2 d}{20}}.
\]
In particular, the right-hand side is smaller than $z/2$ if $\theta^2 d$
is sufficiently large.
\end{pf}

\subsection{The large-\texorpdfstring{$\theta$}{$theta$} case}

We now give an analogue of Proposition~\ref{prop:large-theta} in the
Galton--Watson case.

%
\begin{proposition}\label{prop:large-theta-gw}
For any $0 < \theta^* < 1$, there is some $d^* = d^*(\theta^*)$ such
that for the broadcast process on the Poisson mean $d$ tree
it holds that for all $\theta\ge\theta^*$, $d \ge d^*$, and $k \ge
K(\theta, d, \delta)$,
\[
\E\sqrt{\llvert X_{\rho,k+1} - Y_{\rho,k+1}\rrvert} \le
\tfrac{1}2 \E\sqrt{\llvert X_{\rho,k} - Y_{\rho,k}\rrvert}.
\]
\end{proposition}

This completes the proof of Theorem~\ref{teo:gw-main} (by the same
argument that followed
Proposition~\ref{prop:large-theta}).

\subsubsection{The case where one child has large error}

Our eventual goal is to prove Proposition~\ref{prop:large-theta} by
a similar analysis of the partial derivatives of $g$ that led to the
proof of Proposition~\ref{prop:large-theta}. In this section, however,
we will deal with one case where the derivatives of $g$ cannot be
controlled well. First, we introduce a parameter $\varepsilon= \varepsilon
(d) > 0$
that will be specified later. Next, fix a vertex $u$ and let $\Omega$
be the event that
all children $i$ of $u$ satisfy $\llvert X_{ui,k} - Y_{ui,k}\rrvert
\le\varepsilon$.
On $\Omega$, we will analyze derivatives of $g$; off $\Omega$ we have
the following
lemma (recalling that $D$ is the number of children of $u$).
%
%
\begin{lemma}\label{lem:far-apart}
For any $0 < \theta^* < 1$, there exist $c, C > 0$ such that
if $\eta< c$, $\theta\in[\theta^*, 1)$, and $\theta^2 d > C$
then for any $\varepsilon> 0$ and $k \ge K(\theta, d, \delta)$
\[
\E\bigl(\sqrt{\llvert X_{u,k+1} - Y_{u,k+1}\rrvert}
1_{\Omega^c} \mid D\bigr) \le\frac{C}{\sqrt{\varepsilon}} D e^{-cD} \E
\sqrt{\llvert X_{ui,k} - Y_{ui,k}\rrvert} 1_{\{\llvert X_{ui,k} -
Y_{ui,k}\rrvert > \varepsilon
\}}.
\]
\end{lemma}

\begin{pf}
First, we condition on $D$; we may then write
\[
1_{\Omega^c} \le\sum_{i=1}^D 1_{\{\llvert X_{ui,k} - Y_{ui,k}\rrvert
> \varepsilon
\}}.
\]
Hence,
\begin{eqnarray*}
&& \E\bigl(\sqrt{\llvert X_{u,k+1} - Y_{u,k+1}\rrvert}
1_{\Omega^c} \mid D\bigr)
\\
&&\qquad \le \E\Biggl(\sum_{i=1}^D
\sqrt{\llvert X_{u,k+1} - Y_{u,k+1}\rrvert} 1_{\{
\llvert X_{ui,k} - Y_{ui,k}\rrvert > \varepsilon\}}
\Bigm| D \Biggr)
\\
&&\qquad = D \E\bigl(\sqrt{\llvert X_{u,k+1} - Y_{u,k+1}\rrvert}
1_{\{\llvert X_{ui,k} - Y_{ui,k}\rrvert >
\varepsilon\}} \mid D\bigr),
\end{eqnarray*}
where the equality follows because all the terms in the sum have the
same distribution.
Now we will condition on $X_{ui,k}$ and $Y_{ui,k}$, and we will show that
on the event $\{\llvert X_{ui,k} Y_{ui,k} \rrvert \ge\varepsilon\}$ we have
%
%
\begin{equation}
\label{eq:far-apart-goal} D \E\bigl(\sqrt{\llvert X_{u,k+1} - Y_{u,k+1}
\rrvert} \mid D, X_{ui,k}, Y_{ui,k}\bigr) \le C D
e^{-c D}.
\end{equation}
After bounding
$1 \le\varepsilon^{-1/2} \sqrt{\llvert X_{ui,k} - Y_{ui,k}\rrvert }$
on the event $\{\llvert X_{ui,k} Y_{ui,k} \rrvert \ge\varepsilon\}$
and then integrating out $X_{ui,k}$ and $Y_{ui,k}$, the proof will be complete.

Now we prove~\eqref{eq:far-apart-goal}.
Condition on $\sigma_u$, and suppose without loss of generality that
$\sigma_u = +$.
If $\theta^2 d$ is sufficiently large then
Proposition~\ref{prop:gw-large-expected-mag} implies that (conditioned on
$\sigma_u=+$) every child $j \ne i$ of $u$ independently satisfies
\[
\Pr(X_{uj,k} \ge1 - \eta\mid\sigma_u=+) \ge7/8.
\]
If we condition also on $D$, Hoeffding's inequality implies that there is
a constant $c > 0$ such that with
probability at least $e^{-cD^2}$, at least $3/4$ of $u$'s children $j$
satisfy $X_{uj,k} \ge1 - \eta$. The remaining children (which possibly
include $i$) satisfy $X_{uj,k} \ge-1$, and so on this event
\[
A:= \prod_{j=1}^D \frac{1-\theta X_{uj,k}}{1+\theta X_{uj,k}}
\le\biggl(\frac{1-\theta(1-\eta)}{1+\theta(1-\eta)} \biggr)^{3D/4}
\biggl(\frac{1+\theta}{1-\theta}
\biggr)^{D/4} \le(3\eta)^{3D/4} \eta^{-D/4}.
\]
Now, $X_{u,k+1} = \frac{1-A}{1+A} \ge1-2A$, and so we conclude that
\[
\Pr\bigl(X_{u,k+1} \ge1 - 2 \cdot3^{3D/4} \eta^{D/2}
\mid X_{ui,k}, Y_{ui,k}, \sigma_u=+, D \bigr) \ge1
- e^{-c D^2}.
\]
The previous argument applies equally well with $X$ replaced by $Y$; hence,
the union bound implies
\[
\Pr\bigl(\llvert X_{u,k+1} - Y_{u,k+1}\rrvert\ge4
\cdot3^{3D/4} \eta^{D/2} \mid X_{ui,k},
Y_{ui,k}, \sigma_u=+, D \bigr) \ge1 - 2 e^{-c D^2}.
\]
On the other hand, we always have the bound $\llvert X_{u,k+1} -
Y_{u,k+1}\rrvert
\le2$,
and so
\[
\E\bigl(\sqrt{\llvert X_{u,k+1} - Y_{u,k+1}\rrvert} \mid
X_{ui,k}, Y_{ui,k}, \sigma_u=+, D\bigr) \le2
\cdot3^{3D/8} \eta^{D/4} + 2 \sqrt2 e^{-c D^2}.
\]
Now, if $\eta< c$ for $c$ sufficiently small, the right-hand side is bounded
by $C e^{-cD}$. This proves~\eqref{eq:far-apart-goal} in the case that
$\sigma_u=+$.
To complete the proof, we apply the symmetric argument conditioned on
$\sigma_u = -$.
\end{pf}

\subsubsection{An analogue of Lemma~\texorpdfstring{\protect\ref{lem:expectation-sqrt}}{3.16}}

The proof of Proposition~\ref{prop:large-theta-gw} proceeds by
analysing the derivatives
of the recurrence~\eqref{eq:g-def}. Recalling that these derivatives
involve a large
product, an important ingredient in the analysis is a bound on the expectation
of each term. The following lemma is analogous to Lemma~\ref
{lem:expectation-sqrt}
in the regular case; an important difference is that Lemma~\ref
{lem:expectation-sqrt-gw}
does not improve as $\eta\to0$. In fact, as we remarked after
Assumption~\ref{ass:gw-Y},
we cannot expect such behavior because of the possibility of extinction.

%
\begin{lemma}\label{lem:expectation-sqrt-gw}
For any $0 < \theta^* < 1$, there are some $\lambda= \lambda(\theta
^*) < 1$
and $d^* = d^*(\theta^*)$ such that
for all $\theta\ge\theta^*$, $d \ge d^*$ and $k \ge K(\theta, d,
\delta)$,
\[
\E\biggl(\sqrt{\frac{1-\theta X_{ui,k}}{1+\theta X_{ui,k}}} \Bigm| \sigma
_u = + \biggr) \le
\lambda.
\]
The same holds with $Y$ replacing $X$.
\end{lemma}

We postpone the details of Taylor expansion and approximation to the \hyperref[app]{Appendix},
but we will include here one of the main ingredients of the proof of Lemma~\ref
{lem:expectation-sqrt-gw}. The point is that in the Galton--Watson case (unlike the $d$-ary
case) if $d$ is
fixed and $\eta\to0$ then we cannot expect $X_{\rho,k}$ to be large
(i.e., close to $1$) with probability
converging to 1. It turns out to be enough, however, to show that
$X_{\rho,k}$ is
\emph{nonnegative} with probability converging to 1.

%
\begin{lemma}\label{lem:wrong-sign}
There is a constant $C$ such that if $\theta^2 d \ge C$ then for any
$k \ge K(\theta, d, \delta)$,
%
%
\begin{equation}
\label{eq:wrong-sign} \Pr(X_{u,k} < 0 \mid\sigma_u=+) \le\eta,
\end{equation}
and similarly for $Y$.
\end{lemma}

\begin{pf}
We will give the argument for $X$ only (the argument for $Y$ is identical).
First, note that if $\eta\ge1/12$ then~\eqref{eq:wrong-sign}
follows directly from Proposition~\ref{prop:gw-large-expected-mag} if
$d^*$ is sufficiently large.
Hence, we may assume that $\eta< 1/12$.
Let $p_k = \Pr^+(X_{\rho,k} < 0)$.
Then by Proposition~\ref{prop:gw-large-expected-mag}, if $C$ is
sufficiently large
then $p_k \le1/12$ for $k \ge K(\delta)$.

Let $Z_-$ be the number of children $i$ of the root with $X_{i,k} < 0$ and
$Z_+$ be the number with $X_{i,k} \ge1-\eta$.
Consider the quantity
\[
Z:= \prod_{i=1}^D \frac{1-\theta X_{ui,k}}{1+\theta X_{ui,k}},
\]
and note that $X_{u,k} < 0$ if and only if $Z > 1$. Now, $Z$ is
increasing in each
$X_{ui,k}$, and $Z$ only increases if
we drop some terms $i$ with $X_{ui,k} \ge0$. Hence,
%
%
\begin{equation}
\label{eq:Z-bound} Z \le\biggl(\frac{1 - \theta(1-\eta)}{1+\theta
(1-\eta)} \biggr)^{Z_+} \biggl(
\frac{1+\theta}{1-\theta} \biggr)^{Z_-} \le(3\eta)^{Z_+}
\eta^{-Z_-}.
\end{equation}

Now, by the definition of $p_k$,
%
%
\begin{eqnarray}\label{eq:wrong-sign-bound}
&& \Pr^+(X_{1,k} < 0)
\nonumber\\[-8pt]\\[-8pt]\nonumber
&&\qquad \le\Pr(X_{1,k} < 0 \mid
\sigma_1 = +) + \Pr(\sigma_1 = - \mid
\sigma_\rho= +) = p_k + \eta.
\end{eqnarray}
Conditioned on $\sigma_\rho$ and $D$, $Z_+ - Z_-$ is a sum of i.i.d. variables
with values $1, -1$, and $0$.
Moreover, Proposition~\ref{prop:gw-large-expected-mag} with $d$
sufficiently large implies that the probability of $X_{i,k} \ge1-\eta
$ is at least $5/6$,
while~\eqref{eq:wrong-sign-bound} implies that the probability of
$X_{i,k} < 0$
is at most $p_k + \eta\le1/6$. Hence, Hoeffding's inequality implies that
\[
\Pr^+(Z_+ - Z_- \le D/3 + 1 \mid D) \le C e^{-c D^2},
\]
for universal constants $c, C > 0$.
Note also that if $Z_- = 0$ then $Z \geq1$ and that
in order to have $Z_- > 0$, there must be some $i$ with
$X_{i,k} < 0$. Note also that if $Z_+ - Z_- \ge D/3$ then
$Z \le3^D \eta^{D/3} \le(3/4)^{D/3} < 1$.
Thus, applying a union bound, Hoeffding's inequality, and~\eqref
{eq:wrong-sign-bound},
%
%
\begin{eqnarray}
\label{eq:union-hoeffding} \Pr^{+}(Z>1 \mid D) &\le&\Pr^+(Z_+ - Z_- \le
D/3, Z_- >
0 \mid D)\nonumber
\\
&\le& D \Pr^+(Z_+ - Z_- \le D/3, X_{1,k} < 0 \mid D)
\nonumber\\[-8pt]\\[-8pt]\nonumber
&=& D\Pr^+(Z_+ - Z_- \le D/3 \mid D, X_{1,k} < 0) \Pr^+(X_{1,k}
< 0 \mid D)
\nonumber
\\
&\le& C D e^{-c D^2} (\eta+ p_k).
\nonumber
\end{eqnarray}
Now, if $d$ is large enough (which can be enforced by taking $C$ large) then
$\E D e^{-cD^2} \le\frac{1}4$, which implies that
\[
p_{k+1} = \Pr^+(X_{\rho,k+1} < 0) = \Pr^+(Z > 1) \le
\frac{\eta+
p_k}{4} \le\max\{\eta/2, p_k/2\}.
\]
Recursing with $k$, we see that $\lim_{k\to\infty} \Pr^+(X_{\rho
,k} < 0) \le\eta/2$,
which implies that $\Pr^+(X_{\rho,k} < 0) \le\eta$ for sufficiently
large $k$.
\end{pf}

\subsubsection{Analysis of the derivatives of $g$}

Our goal in this section is the following lemma, for which we recall that
$\Omega$ is the event that all children $i$ of $u$ satisfy $\llvert
X_{ui,k} - Y_{ui,k}\rrvert \le\varepsilon$.
Let $\Omega_i$ be the event that $\llvert X_{ui,k} - Y_{ui,k}\rrvert
\le\varepsilon$.
%
%
\begin{lemma}\label{lem:close-together}
For any $0 < \theta^* < 1$, there are constants $c, C > 0$ such that
for all $0 < \varepsilon< 1/4$,
all $d \ge d^*(\theta^*)$,
and for any $k \ge K(\theta, d, \delta)$,
\[
\E\bigl(1_{\Omega} \sqrt{\llvert X_{u,k+1} - Y_{u,k+1}
\rrvert} \mid D\bigr) \le C D \bigl(\varepsilon^{-1} e^{-cD} +
\sqrt{\varepsilon} \bigr) \E1_{\Omega_i} \sqrt{\llvert X_{ui,k} -
Y_{ui,k}\rrvert}.
\]
\end{lemma}

We begin with an slightly improved version of~\eqref{eq:mvt-g}: since
$\llvert X_{u,k+1} - Y_{u,k+1}\rrvert \le2$, we can trivially
improve~\eqref{eq:mvt-g}
to
%
%
\begin{eqnarray}
\label{eq:mvt-g-improved}
&& \llvert X_{u,k+1} - Y_{u,k+1}\rrvert
\nonumber\\[-8pt]\\[-8pt]\nonumber
&&\qquad \le\sum_{i=1}^D \min\bigl\{2, \llvert
X_{ui,k} - Y_{ui,k}\rrvert\max\bigl\{ h_i(X_{L_1(u),k}),
h_i(Y_{L_1(u),k}) \bigr\} \bigr\}.
\end{eqnarray}
Note that
$1_{\Omega} \le1_{\Omega_i}$ for any $i$ (recall that $\Omega_i = \{
\llvert X_{ui,k} - Y_{ui,k}\rrvert \le\varepsilon\}$),
and so
\begin{eqnarray*}
&& \llvert X_{u,k+1} - Y_{u,k+1}\rrvert1_{\Omega}
\\
&&\qquad \le\sum_{i=1}^D 1_{\Omega_i} \min
\bigl\{2, \llvert X_{ui,k} - Y_{ui,k}\rrvert\max\bigl
\{h_i(X_{L_1(u),k}), h_i(Y_{L_1(u),k}) \bigr\}
\bigr\}.
\end{eqnarray*}
Now, the terms on the right-hand side have identical distributions;
hence, taking conditional expectations gives
\begin{eqnarray*}
&& \E\bigl(\sqrt{\llvert X_{u,k+1} - Y_{u,k+1}\rrvert}
1_\Omega\mid D\bigr)
\\
&&\qquad \le D \E\bigl(1_{\Omega_i} \min\bigl\{2, \sqrt{\llvert X_{ui,k}
- Y_{ui,k}\rrvert\max\bigl\{ h_i(X_{L_1(u),k}),
h_i(Y_{L_1(u),k}) \bigr\}} \bigr\} \mid D \bigr).
\end{eqnarray*}

Defining
\[
Z_X = \min\bigl\{1, \sqrt{\llvert X_{ui,k} -
Y_{ui,k}\rrvert h_i(X_{L_1(u),k})} \bigr\}
\]
and similarly for $Z_Y$,
we see that to prove Lemma~\ref{lem:close-together} it suffices
to show that
\[
\E(1_{\Omega_i} Z_X \mid D) \le C \bigl(\varepsilon^{-1}
e^{-cD} + \sqrt\varepsilon\bigr) \E1_{\Omega_i} \sqrt{\llvert
X_{ui,k} - Y_{ui,k}\rrvert},
\]
and similarly for $Z_Y$.
We will show this by conditioning on $X_{ui,k}$ and $Y_{ui,k}$;
that is, we will show the stronger statement that on the event $\Omega_i$,
%
%
\begin{equation}
\label{eq:small-d-goal-1} \E(Z_X \mid D, X_{ui,k}, Y_{ui,k})
\le C \bigl(\varepsilon^{-1} e^{-cD} + \sqrt\varepsilon\bigr)
\sqrt{\llvert X_{ui,k} - Y_{ui,k}\rrvert}
\end{equation}
(and similarly for $Z_Y$).

We split the analysis of $Z_X$ and $Z_Y$ into two cases. The first case
is the easy
case: if $\eta$ is bounded away from zero or $\llvert X_{ui,k}
\rrvert $ and
$\llvert Y_{ui,k} \rrvert $ are bounded
away from~1 then the denominator in $h_i$ is bounded above.

%
\begin{lemma}\label{lem:close-together-away-from-1}
For any $0 < \theta^* < 1$, there are constants $c, C > 0$ such that
for all $\varepsilon\ge0$,
all $d \ge d^*(\theta^*)$,
and for any $k \ge K(\theta, d, \delta)$,
if $\max\{\llvert X_{ui,k} \rrvert, \llvert Y_{ui,k} \rrvert \} \le1
- \varepsilon$ then
\[
\E(Z_X \mid D, X_{ui,k}, Y_{ui,k}) \le
\frac{C \lambda^{D-1}}{\max\{\sqrt{\eta}, \varepsilon\}} \sqrt{\llvert
X_{ui,k} - Y_{ui,k}\rrvert},
\]
and similarly for $Z_Y$.
\end{lemma}

\begin{pf}
By the definition of $h_i$, and because $\llvert X_{ui,k} \rrvert \le
1-\varepsilon$,
\[
h_i(X_{ui,k}) \le\frac{4}{\max\{\eta, \varepsilon^2\}} \min\biggl\{ \prod
_{j\ne i} \frac{1-\theta X_{uj,k}}{1 + \theta X_{uj,k}}, \prod
_{j\ne i} \frac{1+\theta X_{uj,k}}{1 - \theta X_{uj,k}} \biggr\}.
\]
Conditioning on $\sigma_u = +$ and considering the first term in the minimum,
Lemma~\ref{lem:expectation-sqrt-gw} implies that
\begin{eqnarray*}
&& \E\bigl(\sqrt{\llvert X_{ui,k} - Y_{ui,k}\rrvert
h_i(X_{L_1(u),k})} \mid D, X_{ui,k},
Y_{ui,k}, \sigma_u=+ \bigr)
\\
&&\qquad \le\frac{2 \lambda^{D-1}}{\max\{\sqrt{\eta}, \varepsilon\}} \sqrt{\llvert
X_{ui,k} - Y_{ui,k}
\rrvert}.
\end{eqnarray*}
By symmetry, the same bound holds if we condition on $\sigma_u = -$.
Recalling that $Z_X \le\sqrt{\llvert X_{ui,k} - Y_{ui,k}\rrvert
h_i(X_{L_1(u),k})}$,
this completes the proof for $Z_X$. The exact same argument applies to
$Z_Y$ also.
\end{pf}

If $X_{ui,k}$ and $Y_{ui,k}$ are allowed to be arbitrarily close to 1
and $\eta$
is allowed to be arbitrarily close to zero, then the argument is
somewhat more
tricky. The basic idea is that if $X_{ui,k}$ is close to 1 then $\sigma
_u$ is
very likely to be $+$, in which case the denominator in $h_i^+$ is at
least 1
and so $h_i^+$ is small. Bad things happen if $\sigma_u = -$ because then
we need to consider $h_i^-$, which has a small denominator. However,
this event is very unlikely conditioned on $X_{ui,k}$ being close to 1, and
so its contribution can be controlled.

%
\begin{lemma}\label{lem:close-together-near-1}
For any $0 < \theta^* < 1$, there are constants $c, C > 0$ such that
for all $0 < \varepsilon< 1/4$,
all $d \ge d^*(\theta^*)$,
and for any $k \ge K(\theta, d, \delta)$,
if
$\llvert X_{ui,k} - Y_{ui,k}\rrvert \le\varepsilon$
and $\max\{\llvert X_{ui,k} \rrvert, \llvert Y_{ui,k} \rrvert \} \ge
1 - \varepsilon$ then
\[
\E(Z_X \mid D, X_{ui,k}, Y_{ui,k}) \le C \bigl(
\lambda^{D-1} + \sqrt{\varepsilon} \bigr) \sqrt{\llvert X_{ui,k} -
Y_{ui,k}\rrvert},
\]
and similarly for $Z_Y$.
\end{lemma}

Before proving Lemma~\ref{lem:close-together-near-1},
note that together with Lemma~\ref{lem:close-together-away-from-1} it
proves \eqref{eq:small-d-goal-1}, and hence Lemma~\ref{lem:close-together}.

\begin{pf*}{Proof of Lemma \ref{lem:close-together-near-1}}
Fix $\theta^* \in(0, 1)$ and take $\lambda< 1$ satisfying Lemma \ref
{lem:expectation-sqrt-gw}.
Since $\varepsilon\le1/4$,
it follows that $X_{ui,k}$ and $Y_{ui,k}$ have the same sign. Without loss
of generality, they are both positive; hence,
if $A = (1 - \min\{X_{ui,k}, Y_{ui,k}\})/2$
and $B = (1 - \max\{X_{ui,k}, Y_{ui,k}\})/2$ then
$0 \le B \le A \le\varepsilon$. Note that $\llvert X_{ui,k} -
Y_{ui,k}\rrvert = 2\llvert A - B\rrvert $.
Now,
\[
\Pr(\sigma_{ui} = + \mid X_{ui,k}, Y_{ui,k}) =
\frac{1 + X_{ui,k}}{2} \ge1 - A,
\]
and so
\[
\Pr(\sigma_u = + \mid X_{ui,k}, Y_{ui,k}) \ge1 -
A - \eta.
\]
Since $X_{ui,k}$ is positive,
\[
h_i^+(X_{L_1(u),k}) = \frac{4}{(1+\theta X_{ui,k})^2} \prod
_{j\ne i} \frac{1-\theta X_{uj,k}}{1+\theta X_{uj,k}} \le4 \prod
_{j\ne i}\frac{1-\theta X_{uj,k}}{1+\theta X_{uj,k}}
\]
and similarly for $Y$. By Lemma~\ref{lem:expectation-sqrt-gw},
if $d^*$ is sufficiently large then
%
%
\begin{eqnarray}
\label{eq:small-d-1} &&\E\Bigl( \sqrt{\llvert X_{ui,k} -
Y_{ui,k}\rrvert h_i^+(X_{L_1(u),k})} \mid D,
X_{ui,k}, Y_{ui,k}, \sigma_u = + \Bigr)\nonumber
\\
&&\qquad \le 4 \E\biggl(\sqrt{\llvert X_{ui,k} - Y_{ui,k}\rrvert\prod
_{j\ne i} \frac
{1-\theta X_{uj,k}}{1+\theta X_{uj,k}}} \Bigm| D,
X_{ui,k}, Y_{ui,k}, \sigma_u = + \biggr)
\\
&&\qquad \le 4 \lambda^{D-1} \sqrt{\llvert X_{ui,k} - Y_{ui,k}
\rrvert},
\nonumber
\end{eqnarray}
since the $X_{uj,k}$ are independent conditioned on $\sigma_u$.
On the other hand, since \mbox{$Z_X \ge0$} we have
%
%
\begin{eqnarray} \label{eq:split-sigma_u}
&& \E(Z_X \mid D, X_{ui,k}, Y_{ui,k})\nonumber
\\
&&\qquad \le
\E(Z_X\mid D, X_{ui,k}, Y_{ui,k},
\sigma_u = +)
\nonumber
\\
&&\quad\qquad{}+ \Pr(\sigma_u = - \mid X_{ui,k},
Y_{ui,k}) \E(Z\mid D, X_{ui,k}, Y_{ui,k},
\sigma_u = -)
\\
&&\qquad \le \E(Z_X\mid D, X_{ui,k}, Y_{ui,k},
\sigma_u = +)
\nonumber
\\
&&\quad\qquad{} + (A + \eta) \E(Z_X\mid D, X_{ui,k},
Y_{ui,k}, \sigma_u = -).\nonumber
\end{eqnarray}
By~\eqref{eq:small-d-1}, the first term of~\eqref{eq:split-sigma_u}
is bounded by $4 \lambda^{D-1} \sqrt{\llvert X_{ui,k} -
Y_{ui,k}\rrvert }$.

Next, we consider the second term of~\eqref{eq:split-sigma_u}; we will consider
the coefficients $A$ and $\eta$ separately.
Now, $Z_X \le\sqrt{\llvert X_{ui,k} - Y_{ui,k}\rrvert
h_i^-(X_{L_1(u),k})}$ and
\[
h_i^-(X_{L_1(u),k)}) = \frac{4}{(1 - \theta X_{ui,k})^2}\prod
_{j \ne i} \frac{1+\theta X_{uj,k}}{1-\theta X_{uj,k}} \le\frac{1}{\max
\{\eta, B\}^2} \prod
_{j \ne i} \frac{1+\theta X_{uj,k}}{1-\theta X_{uj,k}}.
\]
Then Lemma~\ref{lem:expectation-sqrt-gw} implies that for $d^*$
sufficiently large,
%
%
\begin{eqnarray} \label{eq:small-d-2}
&& \E\Bigl(\sqrt{h_i^-(X_{L_1(u),k})} \mid D,
X_{ui,k}, Y_{ui,k}, \sigma_u = - \Bigr)\nonumber
\\
&&\qquad \le
\frac{1}{\max\{\eta,B\}} \prod_{j\ne i} \E\biggl(\sqrt{
\frac{1+\theta X_{uj,k}}{1-\theta X_{uj,k}}} \Bigm| D, \sigma_u = - \biggr)
\\
&&\qquad \le \frac{\lambda^{D-1}}{\max\{\eta,B\}}.\nonumber
\end{eqnarray}
In particular, we have
%
%
\begin{eqnarray}\label{eq:eta-term}
&& \eta\E(Z \mid D, X_{ui,k}, Y_{ui,k}, \sigma_u =-)\nonumber
\\
&&\qquad \le \eta\sqrt{\llvert X_{ui,k} - Y_{ui,k}\rrvert} \E\Bigl(
\sqrt{h_i^-(X_{L_1(u),k})} \mid D, X_{ui,k},
Y_{ui,k}, \sigma_i \Bigr)
\\
&&\qquad \le \lambda^{D-1} \sqrt{\llvert X_{ui,k} - Y_{ui,k}
\rrvert},\nonumber
\end{eqnarray}
which handles the term in~\eqref{eq:split-sigma_u} involving $\eta$.

Next, we consider the term involving $A$.
If $A \le2B$ then we may use~\eqref{eq:small-d-2} for the bound
%
%
\begin{equation}
\label{eq:AleB} \E\Bigl(\sqrt{h_i^-(X_{L_1(u),k})}
\mid D, X_{ui,k}, Y_{ui,k}, \sigma_u = - \Bigr) \le
\frac{\lambda^{D-1}}{B} \le\frac{2\lambda^{D-1}}{A}.
\end{equation}
Alternatively, if
$A \ge2B$ then $\llvert X_{ui,k} - Y_{ui,k}\rrvert = 2\llvert A -
B\rrvert \ge A$; since $Z \le1$,
we have
\[
A \E(Z \mid X_{ui,k}, Y_{ui,k}, \sigma_u = -) \le
A \le\sqrt{A \llvert X_{ui,k} - Y_{ui,k}\rrvert} \le\sqrt{
\varepsilon\llvert X_{ui,k} - Y_{ui,k}\rrvert}.
\]
Combining this with~\eqref{eq:AleB}, we have
\[
A \E(Z \mid X_{ui,k}, Y_{ui,k}, \sigma_u = -) \le
\max\bigl\{2\lambda^{D-1}, \sqrt{\varepsilon} \bigr\} \sqrt{\llvert
X_{ui,k} - Y_{ui,k}\rrvert}
\]
in either case. Combining this with~\eqref{eq:eta-term} and going back
to~\eqref{eq:split-sigma_u}, we have
\[
\E(Z \mid D, X_{ui,k}, Y_{ui,k}) \le\bigl(C
\lambda^{D-1} + \sqrt{\varepsilon} \bigr) \sqrt{\llvert X_{ui,k} -
Y_{ui,k}\rrvert},
\]
which completes the proof.
\end{pf*}

\subsubsection{Putting it together}

Finally, we put together the various cases and prove Proposition~\ref
{prop:large-theta-gw}.
First, fix $\theta^*$ and put $\varepsilon= d^{-4}$. The easy case is
when $\eta\ge c$, where $c$ is the constant from Lemma~\ref{lem:far-apart}.
In this case, Lemma~\ref{lem:close-together-away-from-1} with
$\varepsilon=0$
implies that
\[
\E(Z_X \mid D, X_{ui,k}, Y_{ui,k}) \le C
e^{-c D} \sqrt{\llvert X_{ui,k} - Y_{ui,k}\rrvert}
\]
and similarly for $Z_Y$. Taking the expectation over $X_{ui,k}$ and
applying~\eqref{eq:mvt-g-improved} implies that
%
%
\begin{equation}
\label{eq:eta-large} \E\bigl(\sqrt{\llvert X_{u,k+1} - Y_{u,k+1}
\rrvert} \mid D\bigr) \le C D e^{-cD} \E\sqrt{\llvert X_{ui,k}
- Y_{ui,k}\rrvert}.
\end{equation}

Now consider the case where $\eta\le c$. By Lemma~\ref{lem:far-apart}
(recalling that $\varepsilon= d^{-4}$), we have
\[
\E\bigl(\sqrt{\llvert X_{u,k+1} - Y_{u,k+1}\rrvert}
1_{\Omega^c} \mid D\bigr) \le C d^2 D e^{-cD}
\E1_{\Omega_i^c} \sqrt{\llvert X_{ui,k} - Y_{ui,k}\rrvert}.
\]
By Lemma~\ref{lem:close-together}, we have
\[
\E\bigl(\sqrt{\llvert X_{u,k+1} - Y_{u,k+1}\rrvert}
1_{\Omega} \mid D\bigr) \le C \bigl(d^4 D e^{-cD} +
d^{-2} D \bigr) \E1_{\Omega_i} \sqrt{\llvert X_{ui,k} -
Y_{ui,k}\rrvert}.
\]
Putting these together, we have
%
%
\begin{eqnarray}
\label{eq:eta-small}
&& \E\bigl(\sqrt{\llvert X_{u,k+1} - Y_{u,k+1}
\rrvert} \mid D\bigr)
\nonumber\\[-8pt]\\[-8pt]\nonumber
&&\qquad \le C \bigl(d^4 D e^{-cD} +
d^{-2} D \bigr) \E\sqrt{\llvert X_{ui,k} - Y_{ui,k}
\rrvert}.
\end{eqnarray}
Noting that the right-hand side of~\eqref{eq:eta-small} is larger
than the right-hand side of~\eqref{eq:eta-large}, we see
that~\eqref{eq:eta-small} holds without extra conditions on $\eta$.
Finally,\vspace*{1pt} we integrate out $D$ in~\eqref{eq:eta-small}. Since
$D \sim\Pois(d)$, we have $\E D = d$ and $\E D e^{-cD} \le e^{-c' d}$
for some constant $c'$ depending on $c$. In particular, if $d$ is sufficiently
large (depending on $C$ and $c$, which depend in turn on $\theta^*$) then
\[
C \E\bigl(d^4 D e^{-cD} + d^{-2} D \bigr) \le
\tfrac{1}2,
\]
which proves Proposition~\ref{prop:large-theta-gw}.

\section{From trees to graphs}

In this section, we will give our reconstruction algorithm and prove that
it performs optimally. It will be convenient for us to work with block
models on fixed vertex sets instead of random ones; therefore, let
$\calG(V^+, V^-, p, q)$ denote the random graph on the vertices
$V^+ \cup V^-$ where pairs of vertices within $V^+$ or $V^-$ are connected
with probability $p$ and pairs of vertices spanning $V^+$ and $V^-$
are included with probability $q$. Note that if $V^-$ and $V^+$ are chosen
to be a uniformly random partition of $[n]$ then $\calG(V^+, V^-,
\frac{a}n, \frac{b}n)$
is simply $\calG(n, \frac{a}n, \frac{b}n)$.

Let \texttt{BBPartition} denote the algorithm of~\cite{MoNeSl14},
which satisfies
the following guarantee, where $V^i$ denotes $\{v \in V(G): \sigma_v
= i\}$.
%
%
\begin{theorem}\label{teo:black-box}
Suppose that $G \sim\calG(V^+, V^-, \frac{a}n, \frac{b}n)$, where
$\llvert V^+\rrvert + \llvert V^-\rrvert = n + o(n)$,
$\llvert V^+\rrvert - \llvert V^-\rrvert = O(\sqrt n)$ and $(a-b)^2
> 2(a+b)$.
There exists some $0 \le\delta< \frac{1}2$
such that as $n \to\infty$, {\tt BBPartition} a.a.s. produces a
partition $W^+ \cup W^- = V(G)$ such that $\llvert W^+\rrvert =
\llvert W^-\rrvert +
o(n) = \frac{n}2 + o(n)$ and
$\llvert W^+ \symdiff V^{i}\rrvert \le\delta n$ for some $i \in\{+,
-\}$.

Moreover, {\tt BBPartition} runs in time $O(n^{1 + o(1)})$.
\end{theorem}

%
\begin{remark}
We should point out that~\cite{MoNeSl14} only claims Theorem~\ref
{teo:black-box} when
$V^+$ and $V^-$ are uniformly random partitions of $[n]$; however, one
easily deduces
the result for almost-balanced partitions from the result for uniformly random
partitions: choose $\varepsilon> 0$ so that $\frac{(a-b)^2}{2(a+b)} >
\frac{1}{1-\varepsilon}$.
Given a graph $G$ from $\calG(V^+, V^-, \frac{a}n, \frac{b}n)$, let $H$
be the graph
obtained by deleting all but $\lceil(1-\varepsilon) n \rceil$ vertices
at random from $G$.
If $(W^+, W^-)$ is the
partition of $H$ according to its vertex labels then one can check that the
sizes of $W^+$ and $W^-$ are contiguous with the sizes of a uniformly random
partition of $\lceil(1-\varepsilon) n\rceil$. Hence, the distribution
of $H$ is contiguous
with $\calG(\lceil(1-\varepsilon) n \rceil, \frac{a}n, \frac{b}n)$. The
results of~\cite{MoNeSl14}
then imply that the labels of $H$ can be recovered adequately (i.e., as
claimed in Theorem~\ref{teo:black-box});
by randomly labeling the
vertices of $G$ that were deleted, we recover Theorem~\ref
{teo:black-box} as stated.
\end{remark}

Note that by symmetry, Theorem~\ref{teo:black-box} also
implies that $\llvert W^- \symdiff V^j\rrvert \le\delta n$ for $j \ne
i \in\{+, -\}$.
In other words, {\tt BBPartition} recovers the correct partition up
to a relabeling of the classes and an error bounded away from $\frac
{1}2$. Note that
$\llvert W^+ \symdiff V^{i}\rrvert = \llvert W^- \symdiff
V^j\rrvert $.
Let $\delta(G)$ be the (random) fraction of vertices that are mislabeled.

For $v \in G$ and $R \in\N$, define $B(v, R) = \{u \in G: d(u,v) \le
R\}$
and $S(v,R) = \{u \in G: d(u,v) = R\}$.
If $B(v, R)$ is a tree (which it is a.a.s.), and $\tau$ is a
labeling $\tau$ on its leaves,
we consider the following estimator of $v$'s label: first, take
$K$ large enough so that Proposition~\ref{prop:gw-large-expected-mag}
holds for
$k = K$.
For $u \in S(v, R-K)$, define $Y_{u,K}(\tau)$
as the sign of $S'_k(\tau)$, where $S'_k$ is given as in the proof
of Proposition~\ref{prop:gw-large-expected-mag}.
That is, $Y_{u,K}(\tau)$ is the sign of a weighted sum of the
labeling $\tau$ on $S(v,R)$.
For $k > K$ and $u \in B(v, R-k)$, define $Y_{u,k}(\tau)$ recursively
by $Y_{u,k} = g(Y_{L_1(u), k-1})$, where $g$ is given by~\eqref{eq:g-def}.
Then $Y$ satisfies Assumption~\ref{ass:gw-Y}.

We remark that the reason for taking this two-stage definition of $Y$ is
because we do not necessarily know how much noise there is on the leaves
(i.e., $\delta$), and so we cannot define $Y$
by~\eqref{eq:Y-def}. Defining $Y$ as we have done avoids the need to
know~$\delta$,
while still satisfying the required assumptions.

Before presenting the algorithm, we will mention one issue that we glossed
over in our earlier sketch: since we will run the black-box algorithm
several times, and since the labels $+$ and $-$ are symmetric, we need some
way to break the symmetry between the various runs of the algorithm. We
do this
by holding out a single vertex of high degree (that we call $u_*$) and breaking
symmetry according to the sign of most of its neighbors.

\begin{algorithm}
\caption{Optimal graph reconstruction algorithm}
\label{alg}
\begin{algorithmic}[1]
\State$R \leftarrow\lfloor\frac{1}{20(a+b)} \log n \rfloor$
\State Take $U \subset V$ to be a random subset of size $\lfloor\sqrt
{n} \rfloor$
\State Let $u_* \in U$ be a random vertex in $U$ with at least $\sqrt
{\log n}$ neighbors in $V \setminus U$ \label{alg:u}
\State$W^+_*, W^-_* \leftarrow\varnothing$
\For{$v \in V \setminus U$}
\State{$W^+_v, W^-_v \leftarrow{\tt BBPartition}(G \setminus B(v,
R-1) \setminus U)$}
\label{alg:v-part}
\If{$a > b$}
\State{relabel $W^+_v, W^-_v$ so that $u_*$ has more neighbors in
$W^+_v$ than $W^-_v$ \label{alg:matching}}
\Else
\State{relabel $W^+_v, W^-_v$ so that $u_*$ has more neighbors in
$W^-_v$ than $W^+_v$ \label{alg:matching-alt}}
\EndIf
\State{Define $\xi\in\{+, -\}^{S(v,R)}$ by $\xi_u = i$ if $u \in
W^i_v$ \label{alg:xi-def}}
\State{Add $v$ to $W^{\sgn(Y_{v,R}(\xi))}_*$ \label{alg:label}}
\EndFor
\For{$v \in U$}
\State{Assign $v$ to $W^+_*$ or $W^-_*$ uniformly at random}
\EndFor
\end{algorithmic}
\end{algorithm}

%
\begin{remark}
Our analysis of Algorithm~\ref{alg} will assume that we can compute
with arbitrary
precision numbers in constant time. However, Propositions~\ref
{prop:large-d-recursion-gw}
and~\ref{prop:large-theta-gw} can also be used to analyze an
implementation of
Algorithm~\ref{alg} with finite-precision arithmetic. Indeed, the only
part of
Algorithm~\ref{alg} where continuous quantities appear is in the computation
of $Y_{v,R}$, and the main question in the computation of $Y_{v,R}$ is
whether the
numerical errors accumulate as we repeatedly apply the recursion $g(x)$
defined in~\eqref{eq:g-def}.

Consider the following finite-precision implementation of the
recursion: first, compute
$\widehat Y_{ui,k}$ to the desired precision for all children $i$ of $u$.
Then compute $g(\widehat Y_{u,L_1(k)})$ to arbitrary precision, and finally
define $\widehat Y_{u,k}$ to be
$g(\widehat Y_{u,L_1(k)})$ truncated to the desired precision. Let us
see what
Proposition~\ref{prop:large-d-recursion-gw} has to say about this procedure
(Proposition~\ref{prop:large-theta-gw} has similar consequences for
the other range of parameters):
if $X$ denotes the true magnetizations and the rounding error is
bounded by $\varepsilon$ then
\begin{eqnarray*}
\E(X_{u,k+1} - \widehat Y_{u,k+1})^2 &\le&\E
\bigl(X_{u,k+1} - g(\widehat Y_{L_1(u),k}) + \varepsilon
\bigr)^2
\\
&\le& O(\varepsilon) + \E\bigl(X_{u,k+1} - g(\widehat Y_{L_1(u),k})
\bigr)^2
\\
&\le& O(\varepsilon) + \tfrac{1}2 \E(X_{u,k} - \widehat
Y_{u,k})^2,
\end{eqnarray*}
which implies that the asymptotic accuracy of our finite-precision
scheme is within
$O(\sqrt\varepsilon)$ of optimal.
\end{remark}

As presented, our algorithm is not particularly efficient (although it
does run in polynomial time)
because we need to re-run \texttt{BBPartition} for almost every vertex
in $V$. However,
one can modify Algorithm~\ref{alg} to run in $O(n^{1+o(1)})$ time by
processing $o(n)$ vertices
in each iteration (a similar idea is used in~\cite{MoNeSl14}). Since
vanilla belief propagation is much more efficient than
Algorithm~\ref{alg} and reconstructs (in practice) just as well, we
have chosen not to present
the faster version of Algorithm~\ref{alg}.

%
\begin{theorem}\label{teo:alg-works}
Algorithm~\ref{alg} produces a partition $W^+_* \cup W^-_* = V(G)$
such that
a.a.s. $\llvert W^+_* \symdiff V^i\rrvert \le(1+o(1)) n (1 - p_T(a,b))$
for some $i \in\{+, -\}$.
\end{theorem}

Theorem~\ref{teo:mns:13} implies that for any algorithm,
$\llvert W^+_* \symdiff V^i\rrvert \ge(1-o(1)) n(1 - p_T(a,b))$
a.a.s. Hence, it is enough
to show that $\E\llvert W^+_* \symdiff V^i\rrvert \le(1 + o(1))
n(1-p_T(a,b))$.
Since Algorithm~\ref{alg}
treats every node equally, it is enough to show that there is some $i$
such that for every $v \in V^i$,
%
%
\begin{equation}
\label{eq:graph-tree} \Pr\bigl(v \in W^+_* \bigr) \to p_T(a,b).
\end{equation}
Moreover, since $\Pr(v \in U) \to0$, it is enough to show~\eqref
{eq:graph-tree}
for every $v \in V^i \setminus U$.

The proof of~\eqref{eq:graph-tree} will take the remainder of this section.
First, we will deal with a technicality: in line~\ref{alg:v-part}, we
are applying
\texttt{BBPartition} to the subgraph of $G$ induced by $V \setminus
B(v, R-1) \setminus U$;
call this graph $G_v$.
We need to justify the fact that $G_v$ satisfies the requirements of
Theorem~\ref{teo:black-box}. Now, if $W^+ = V^+ \setminus B(v, R-1)
\setminus U$
and $W^- = V^- \setminus B(v, R-1) \setminus U$ then
$G_v \sim\calG(W^+, W^-, \frac{a}n, \frac{b}n)$. Since
\[
\bigl\llvert W^+\bigr\rrvert+ \bigl\llvert W^-\bigr\rrvert= n - \bigl
\llvert
B(v, R-1)\bigr\rrvert- \lfloor\sqrt n \rfloor
\]
and
\begin{eqnarray*}
\bigl\llvert\bigl\llvert W^+\bigr\rrvert- \bigl\llvert W^-\bigr
\rrvert\bigr
\rrvert&\le&\bigl\llvert\bigl\llvert V^+\bigr\rrvert- \bigl\llvert
V^-\bigr
\rrvert\bigr\rrvert+ \bigl\llvert B(v, R-1)\bigr\rrvert+ \lfloor
\sqrt n \rfloor
\\
&\le& O(\sqrt n) + \bigl\llvert B(v, R-1)\bigr\rrvert,
\end{eqnarray*}
we see that the hypothesis of Theorem~\ref{teo:black-box} is satisfied
as long as
$\llvert B(v, R-1)\rrvert = O(\sqrt n)$. This is indeed the case;
Lemma 4.4 of~\cite
{MoNeSl13} shows that
$\llvert B(v, R)\rrvert = O(n^{1/8})$ for the value of $R$ that we
have chosen.
%
%
\begin{lemma}\label{lem:small-ball}
$\llvert B(v, R)\rrvert = O(n^{1/8})$ a.a.s.
\end{lemma}

We conclude, therefore,
that Theorem~\ref{teo:black-box} applies in line~\ref{alg:v-part} of
Algorithm~\ref{alg}.
%
%
\begin{lemma}\label{lem:correlated}
There is some $0 \le\delta< \frac{1}2$ such that
for any $v \in V \setminus U$, there a.a.s. exists some $i \in\{+, -\}
$ such that
$\llvert W_v^+ \symdiff V^i\rrvert \le\delta n$, with $W_v^+$ defined
as in
line~\ref{alg:v-part}.
\end{lemma}

\subsection{Aligning the calls to {\tt BBPartition}}
Next, let us discuss in more detail the purpose of $u_*$ and line~\ref
{alg:matching}. Recall that
Algorithm~\ref{alg} relies on multiple applications of \texttt
{BBPartition}, each of which is
only guaranteed to give a good labeling up to swapping $+$ and $-$. In
order to get a consistent
labeling at the end, we need to ``align'' these multiple applications
of \texttt{BBPartition}.

We will break the symmetry between $+$ and $-$ by assuming, from now
on, that $u_*$ is
labeled $+$.
Next, let us note some properties of $u_*$.
%
%
\begin{lemma}\label{lem:u*}
In line~\ref{alg:u}, there a.a.s. exists at least one $u \in U$ with
more than
$\sqrt{\log n}$ neighbors in $V \setminus U$; hence, $u_*$ is well defined.
Moreover, there is some $\eta> 0$ such that a.a.s. at least a $(1 +
\eta)/2$-fraction
of $u_*$'s neighbors
in $V\setminus U$ either are labeled $+$ (if $a > b$) or $-$ (if $a < b$).
Finally, for any $v \in V \setminus U$, $u_*$ a.a.s. has no neighbors
in $B(v, R-1)$.
\end{lemma}

\begin{pf}
For the first claim, note that every $u \in U$ independently has more than
$\Binom(\lceil n (1-\varepsilon/2)\rceil, \frac{\min\{a,b\}}{n})$
neighbors in $V \setminus U$,
and the maximum of $\sqrt n$ such variables is of order $\Theta(\log n
/ \log\log n) \gg\sqrt{\log n}$.

For the second claim, let $d$ be the number of neighbors that $u_*$ has
in $V \setminus U$
and note that $d = O(\log n)$ a.a.s., because the maximum degree of any
vertex in $G$
is $O(\log n)$. Conditioned on $d$, the number of $u_*$'s $+$-labeled neighbors
in $V \setminus U$ is dominated by $\Binom(d, \frac{a}{a+b} \cdot
\frac{\llvert V^+\rrvert -d}{\llvert V^-\rrvert })$;
this is because the neighborhood of $u_*$ may be generated by
sequentially choosing $d$
neighbors without replacement from $V \setminus U$, where a
$+$-labeled neighbor
is chosen with probability $\frac{a}{a+b}$ times the fraction of
$+$-labeled vertices
remaining. Since $\llvert V^+\rrvert = n/2 \pm O(n^{1/2})$ and $d =
o(n)$, we see that
$u_*$ a.a.s. has at least $d(\frac{a}{a+b} - o(1))$ $+$-labeled
neighbors. If $a > b$,
then this verifies the second claim; if $a < b$, then we repeat the
argument with $+$
replaced by $-$.

For the final claim, note that if $u_*$ has a neighbor in $B(v, R-1)$
then $u_* \in B(v, R)$.
But (by Lemma~\ref{lem:small-ball}) $\llvert B(v, R)\rrvert =
O(n^{1/8})$ a.a.s.,
and so
with probability tending to 1, $B(v, R)$ does not intersect $U$ at all;
in particular,
it does not contains~$u_*$.
\end{pf}

From now on, suppose without loss of generality that $\sigma_{u^*} = +$.
Thanks to the previous paragraph and Theorem~\ref{teo:black-box}, we
see that the
relabeling in lines~\ref{alg:matching} and~\ref{alg:matching-alt}
correctly aligns
$W_v^+$ with $V^+$.
%
%
\begin{lemma}\label{lem:correlated-aligned}
There is some $0 \le\delta< \frac{1}2$ such that for any $v \in V
\setminus U$,
$\llvert W_v^+ \symdiff V^+\rrvert \le\delta n$ a.a.s., with $W_v^+$
defined as in
line~\ref{alg:matching} or line~\ref{alg:matching-alt}.
\end{lemma}

\begin{pf}
Assume for now that $a > b$.
Just for the duration of this proof, let $W_v^+$ and $W_v^-$ denote the
partition as defined in
line~\ref{alg:v-part} of Algorithm~\ref{alg}, while $\widetilde W_v^+$
and $\widetilde W_v^-$ denote the partition
defined by line~\ref{alg:matching} or line~\ref{alg:matching-alt}.

Recall from Lemma~\ref{lem:u*} that $u_*$ has at least $\sqrt{\log
n}$ neighbors in
$V \setminus B(v, R-1) \setminus U$, of which at least a $(1 + \eta
)/2$-fraction are labeled $+$;
let $d \ge\sqrt{\log n}$ be the number of neighbors that $u_*$ has in
$V \setminus B(v, R-1) \setminus U$, and
let $p \ge(1 + \eta)/2$ be the fraction that are actually labeled $+$.
Note that the labeling $W_v^+, W_v^-$ produced in line~\ref
{alg:v-part} is independent
of the set of $u_*$'s neighbors in $V \setminus B(v, R-1) \setminus U$,
because $W_v^+$ and $W_v^-$ depend only on
edges within $V \setminus B(v, R-1) \setminus U$
and these are independent of the edges adjoining $u_*$. That is,
conditioned on $d$, $p$, $W_v^+$ and $W_v^-$,
the neighbors of $u_*$ can be generated by taking $u_*$'s $+$-labeled
neighbors to be a uniformly random set of
$pd$ $+$-labeled vertices and then taking $u_*$'s $-$-labeled
neighbors to be a uniformly random set
of $(1-p)d$ $-$-labeled vertices.
Hence, if $N_{ij}$ (for $i, j \in\{+, -\}$) is the number
of $u_*$'s neighbors in $V^i \cap W_v^j$ then conditioned on $d$, $p$
and $W_v^+$,
$N_{++}$ is distributed as $\hypergeom(dp, \llvert W_v^+ \cap
V^+\rrvert, \llvert V^+\rrvert )$
and $N_{-+}$ is distributed as $\hypergeom(d(1-p), \llvert W_v^+ \cap
V^-\rrvert, \llvert V^-\rrvert )$.
Since $d = o(\llvert V^+\rrvert ) = o(\llvert V^-\rrvert )$ and $d
\to\infty$ a.a.s., we have
\begin{eqnarray*}
N_{++} &\ge&\bigl(1 - o(1) \bigr) dp \frac{\llvert W_v^+ \cap
V^+\rrvert }{\llvert V^+\rrvert } =
\bigl(1-o(1) \bigr) \frac{2dp \llvert W_v^+ \cap V^+\rrvert }{n},
\\
N_{-+} &\ge&\bigl(1 - o(1) \bigr) d(1-p) \frac{\llvert W_v^+ \cap
V^-\rrvert }{\llvert V^-\rrvert } =
\bigl(1-o(1) \bigr) \frac{2d(1-p) \llvert W_v^+ \cap V^-\rrvert }{n}.
\end{eqnarray*}
Adding these together, we have
%
%
\begin{equation}
\label{eq:total+} N_{++} + N_{-+} = \bigl(1-o(1) \bigr)
\frac{d}{n} \bigl(\alpha+ \beta+ (2p-1) (\alpha- \beta) \bigr),
\end{equation}
where $\alpha= \llvert W_v^+ \cap V^+\rrvert $ and $\beta= \llvert
W_v^+ \cap V^-\rrvert $.

Now, Lemma~\ref{lem:correlated} admits two cases: if $i = +$ then
\begin{eqnarray*}
\delta n &\ge& \bigl\llvert W_v^+ \symdiff V^+\bigr\rrvert= \bigl
\llvert W_v^+ \cap V^-\bigr\rrvert+ \bigl\llvert W_v^-
\cap V^+\bigr\rrvert
\\
&=& \bigl\llvert W_v^+ \cap V^-\bigr\rrvert+
\frac{n}2 + o(n) - \bigl\llvert W_v^+ \cap V^+\bigr\rrvert,
\end{eqnarray*}
and we conclude that
$\alpha- \beta\ge(\frac{1}2 - \delta- o(1))) n$.
A similar argument when $i = -$ in Lemma~\ref{lem:correlated}
shows that in that case $\alpha- \beta\le-(\frac{1}2 - \delta- o(1)) n$.
In either case, $\alpha+ \beta= (1 + o(1)) n/2$.

If $i=+$ in Lemma~\ref{lem:correlated},
then since $p - 1/2 \ge\eta/2$,~\eqref{eq:total+} implies
\[
N_{++} + N_{-+} = \bigl(1 - o(1) \bigr) d \biggl(
\frac{1}2 + \frac{(\sfrac{1}2 -
\delta) \eta}{2} \biggr)
\]
a.a.s. Since $N_{++} + N_{-+} + N_{+-} + N_{--} = d$, we have in particular
$N_{++} + N_{-+} > N_{+-} + N_{--}$ a.a.s., and so $u_*$ has most of
its neighbors in $W_v^+$.
Hence, $\widetilde W_v^+ = W_v^+$ and so Lemma~\ref{lem:correlated} with
$i=+$ implies the conclusion
of Lemma~\ref{lem:correlated-aligned} holds.
On the other hand, if $i=-$ in Lemma~\ref{lem:correlated} then $\alpha
- \beta< -(\frac{1}2 - \delta) n$;
by~\eqref{eq:total+}, $N_{+-} + N_{--} > N_{++} + N_{-+}$. Then $u_*$
has most of its neighbors in
$W_v^-$ and so $\widetilde W_v^+ = W_v^-$. By Lemma~\ref{lem:correlated}
with $i=-$, the conclusion
of Lemma~\ref{lem:correlated-aligned} holds.

Finally, we mention the case $a < b$: essentially the same argument
holds except that instead
of $p \ge(1 + \eta)/2$ we have $p \le(1-\eta)/2$. Then $i=+$
implies that $u_*$ has most of
its neighbors in $W_v^-$, while $i=-$ implies that $u_*$ has most of
its neighbors
in $W_v^+$.
\end{pf}

\subsection{Calculating $v$'s label}
To complete the proof of~\eqref{eq:graph-tree}
(and hence Theorem~\ref{teo:alg-works}), we need to discuss the
coupling between
graphs and trees. We will invoke a lemma from~\cite{MoNeSl13} which
says that a neighborhood in $G$ can be coupled with a multi-type
branching process of the sort that we considered in Section~\ref{sec:gw}.
Indeed, let $T$ be the Galton--Watson tree of Section~\ref{sec:gw}
[with $d = (a+b)/2$] and let $\sigma'$ be a labeling on it, given by running
the two-state broadcast process with parameter $\eta= b/(a+b)$.
We write $T_R$ for $T \cap\N^R$; that is, the part of $T$ which has
depth at most $R$.

%
\begin{lemma}\label{lem:coupling}
For any fixed $v \in G$,
there is a coupling between $(G, \sigma)$ and $(T, \sigma')$ such that
$(B(v, R), \sigma_{B(v,R)}) = (T_R, \sigma'_{T_R})$ a.a.s.
\end{lemma}

Armed with Lemma~\ref{lem:coupling}, we will consider a slightly different
method of generating $G$, which is nevertheless equivalent to the original
model in the sense that the new method and the old method may be coupled
a.a.s.
In the new construction, we begin by assigning labels to $V(G)$
uniformly at
random. Beginning with a fixed vertex $v$, we construct $B(v,R-1)$
by drawing a Galton--Watson
tree of depth $R-1$ rooted at $v$, with labels distributed according to the
broadcast process.
On the vertices that remain [i.e., those that were not used in $B(v,R-1)$],
we construct a graph $G'$ according to the stochastic block model
with parameters $a/n$ and $b/n$. Finally, we join $B(v,R-1)$
to the rest of the graph: for every vertex $u \in S(v,R-1)$,
we draw $\Pois(a/(a+b))$ vertices at random from $G'$ with label
$\sigma_u$
and $\Pois(b/(a+b))$ vertices from $G'$ with label $-\sigma_u$;
we connect all these vertices to $u$.
It follows from Lemma~\ref{lem:coupling} that this construction is equivalent
to the original construction. It also follows from Lemma~\ref{lem:small-ball}
that $\llvert G'\rrvert \ge n - O(n^{1/8})$ a.a.s.

The advantage of the construction above is that it becomes obvious that
the edges of
$G' = G \setminus B(v,R-1) \setminus U$ are independent of both
$B(v,R-1)$ and the edges
joining $B(v,R-1)$ to $G'$. Since $W_v^+$ and $W_v^-$ are both functions
of $G'$ only, it follows that $B(v,R-1)$ and its edges to $G'$ are also
independent of $W_v^+$ and $W_v^-$. Using this observation, we can
improve Lemma~\ref{lem:coupling} to include the noisy labels. In particular,
we claim that the labeling $\xi$ produced in line~\ref{alg:xi-def}
of Algorithm~\ref{alg} has the same distribution as the noisy labeling
$\tau$ of the noisy broadcast process.

In view of Lemma~\ref{lem:coupling}, it suffices to condition on
$\sigma$,
$B(v,R-1)$ and $G'$, and to show that the conditional distribution of
$\xi$ is essentially the same as the conditional distribution of
$\tau$ given $T$ and $\sigma'$ in the noisy broadcast process.
Since the edges joining $B(v, R-1)$ to $G'$ are independent of $W_v^+$
and $W_v^-$, for any $u \in S(v, R-1)$ with $\sigma_u=+$ we have
\begin{eqnarray*}
\# \bigl\{w \sim u: w \in G', \sigma_w = +,
\xi_w = + \bigr\} &\sim&\Binom\biggl(\bigl\llvert V^+ \cap
W_v^+\bigr\rrvert, \frac{a}{n} \biggr),
\\
\# \bigl\{w \sim u: w \in G', \sigma_w = +,
\xi_w = - \bigr\} &\sim&\Binom\biggl(\bigl\llvert V^+ \cap
W_v^-\bigr\rrvert, \frac{a}{n} \biggr),
\\
\# \bigl\{w \sim u: w \in G', \sigma_w = -,
\xi_w = - \bigr\} &\sim&\Binom\biggl(\bigl\llvert V^- \cap
W_v^-\bigr\rrvert, \frac{b}{n} \biggr),
\\
\# \bigl\{w \sim u: w \in G', \sigma_w = -,
\xi_w = + \bigr\} &\sim&\Binom\biggl(\bigl\llvert V^- \cap
W_v^+\bigr\rrvert, \frac{b}{n} \biggr).
\end{eqnarray*}
Moreover, the random variables above are independent as $u$ ranges
over $S(v,R-1)$. Now, if we define $\delta= \frac{1}n \llvert V^+
\symdiff W_v^+\rrvert $
then $\Binom(\llvert V^+ \cap W_v^+\rrvert, a/n)$ and $\Pois(a(1-\delta
)/2)$ are
at total variation distance at most $O(n^{-1/2})$; here, we are using
the fact that $\llvert V^+ \cap W_v^+\rrvert = (1-\delta)n/2 \pm
O(n^{1/2})$, which
follows because $V^+, V^-$ are an equipartition of $V(G)$ and
$W_v^+, W_v^-$ are an equipartition of $V(G')$, which contains
all but at most $O(\sqrt n)$ vertices of $G$.
Similarly, we have
\begin{eqnarray*}
\# \bigl\{w \sim u: w \in G', \sigma_w = +,
\xi_w = + \bigr\} & \stackrel{d} {\approx} &\Pois\bigl(a(1-\delta)/2
\bigr),
\\
\# \bigl\{w \sim u: w \in G', \sigma_w = +,
\xi_w = - \bigr\} & \stackrel{d} {\approx} &\Pois(a\delta/2),
\\
\# \bigl\{w \sim u: w \in G', \sigma_w = -,
\xi_w = - \bigr\} & \stackrel{d} {\approx} &\Pois\bigl(b(1-\delta)/2
\bigr),
\\
\# \bigl\{w \sim u: w \in G', \sigma_w = -,
\xi_w = + \bigr\} & \stackrel{d} {\approx} &\Pois(b\delta/2),
\end{eqnarray*}
where ``$\stackrel{d}{\approx}$'' means that the distributions
are at total variation distance at most $O(n^{-1/2})$. Note
that the distributions on the right-hand side are exactly the distributions
of the noisy labels $\tau$ under the noisy broadcast process.
By a similar argument for $\sigma_u=-$, and a union bound over
the $O(n^{1/8})$ choices for $u$,
we see that the joint distribution of $B(v, R)$ and $\{\xi_u: u \in
S(v,R)\}$
a.a.s. the same as the joint distribution of $T_R$ and $\{\tau_u: u
\in
\partial T_R\}$. Hence, by Theorem~\ref{teo:gw-main},
\[
\lim_{n \to\infty} \Pr\bigl(Y_{v,R}(\xi) =
\sigma_v \bigr) = p_T(a, b).
\]
By line~\ref{alg:label} of Algorithm~\ref{alg}, this completes the
proof of~\eqref{eq:graph-tree}.

\setcounter{equation}{0}
\begin{appendix}
\section*{Appendix: Bounds on \texorpdfstring{$\mathbb{E}\sqrt{\frac{1-\theta X}{1+\theta X}}$}{$\mathbb{E}sqrt{\frac{1-theta X}{1+theta X}}$}}\label{app}
Because of the form of the recursion~\eqref{eq:g-def}, at various points
in our analysis we require bounds on quantities of the form
$\E\sqrt{\frac{1-\theta X}{1+\theta X}}$, under various
assumptions on~$X$. These estimates are
elementary but tedious to check, and so we have collected them here.

\begin{pf*}{Proof of Lemma~\ref{lem:expectation-sqrt}}
By Lemma~\ref{lem:large-mag}, we have
\begin{eqnarray*}
\Pr(X_{ui,k} \ge1 - \eta\alpha t \mid\sigma_u = +) &\ge& \Pr
(X_{ui,k} \ge1 - \eta\alpha t \mid\sigma_{ui} = +) - \eta
\\
&\ge& 1 - t^{-1} - \eta,
\end{eqnarray*}
where $\alpha= C / (\theta^2 d)$ can be taken arbitrarily small if we
require $\theta^2 d$
to be large.

Fix some $\varepsilon= \varepsilon(\theta^*) > 0$ to be determined later. Take
$t = \varepsilon^{-1} \eta^{-3/4}$ so that
\[
\Pr\biggl(X \ge1 - \frac{\alpha\eta^{1/4}}{\varepsilon} \biggr) \ge1 -
\varepsilon
\eta^{3/4} - \eta.
\]
Now, suppose that $\alpha$ is small enough so that
$\alpha\varepsilon^{-1} \le\varepsilon$. Then
%
%
\begin{equation}
\label{eq:expectation-sqrt-large-mag} \Pr\bigl(X \ge1 - \varepsilon\eta
^{1/4} \bigr) \ge1 -
\varepsilon\eta^{3/4} - \eta.
\end{equation}

Now consider the function
\[
f(x):= \sqrt{\frac{1-\theta x}{1+\theta x}}.
\]
Note that $f(x)$ is decreasing in $x$, and hence
\[
\E f(X) \le f(s) \Pr(X \ge s) + f(-1) \Pr(X \le s),
\]
for any random variable $X$ supported on $[-1, 1]$ and for any $s \in
[-1, 1]$.
Applying this for $s = 1 - \varepsilon\eta^{1/4}$, we have
[by~\eqref{eq:expectation-sqrt-large-mag}]
%
%
\begin{equation}
\label{eq:expectation-decomposed} \E f(X) \le f \bigl(1 - \varepsilon\eta
^{1/4} \bigr)
\bigl(1 - \varepsilon\eta^{3/4} - \eta\bigr) + f(-1) \bigl(\varepsilon
\eta^{3/4} + \eta\bigr).
\end{equation}
We will now check that if $\eta\le\frac{1-\theta^*}{2} < 1/2$ then
each term on the right-hand side of~\eqref{eq:expectation-decomposed}
can be made strictly smaller than $1/2$,
and also smaller than $2\eta^{1/4}$,
by taking $\varepsilon= \varepsilon(\theta^*)$ small enough.
This will complete the proof of the lemma.

We consider the term involving $f(-1)$ first:
%
%
\begin{equation}
\label{eq:expectation-decomposed-1} f(-1) \bigl(\varepsilon\eta^{3/4} +
\eta\bigr) = \varepsilon
\eta^{1/4} \sqrt{1-\eta} + \sqrt{\eta(1-\eta)}.
\end{equation}
On the interval $\eta\in[0, \frac{1-\theta^*}{2}]$, $\sqrt{\eta
(1-\eta)}$
is bounded away from $1/2$, and $\eta^{1/4}{\sqrt{1-\eta}}$ is bounded
above. Hence,~\eqref{eq:expectation-decomposed-1} is bounded away from $1/2$
as long as $\varepsilon(\theta^*)$ is small enough.
On the other hand,~\eqref{eq:expectation-decomposed-1}
is also bounded by $2\eta^{1/4}$ as long as $\varepsilon\le1$.

Next, we consider the $f(1-\varepsilon\eta^{1/4})$ term
of~\eqref{eq:expectation-decomposed}. Note that
$\theta(1 - \varepsilon\eta^{1/4}) \ge1 - 2\eta- \varepsilon\eta
^{1/4}$ and so
\[
f \bigl(1 - \varepsilon\eta^{1/4} \bigr) \le\sqrt{\frac{2\eta+ \varepsilon\eta^{1/4}}{
2 - (2 \eta+ \varepsilon\eta^{1/4})}} \le
\sqrt{\frac{\eta}{1-\eta}} + C \varepsilon\eta^{1/4},
\]
where the second inequality follows from applying a first-order Taylor
expansion to
the function $\sqrt{x/(1-x)}$ near $x=\eta$. Here, $C$ is a universal constant
because the assumptions $\eta\le1/2$ and $\varepsilon\le1$ ensure that
the derivative
of $\sqrt{x/(1-x)}$ is universally bounded on the interval of
interest. Thus,
%
%
\begin{eqnarray} \label{eq:expectation-decomposed-2}
f \bigl(1-\varepsilon\eta^{1/4} \bigr) \bigl(1 - \varepsilon
\eta^{1/4} - \eta\bigr) &\le& f \bigl(1-\varepsilon\eta^{1/4}
\bigr) (1-\eta)
\nonumber\\[-8pt]\\[-8pt]\nonumber
&\le&\sqrt{\eta(1-\eta)} + C \varepsilon\eta^{1/4} (1-\eta).\nonumber
\end{eqnarray}
As before, on the interval $\eta\in[0, \frac{1-\theta^*}{2}]$,
$\sqrt{\eta(1-\eta)}$ is bounded away from $1/2$,
and $\eta^{1/4}(1-\eta)$ is bounded
above. Hence,~\eqref{eq:expectation-decomposed-2} is bounded away from $1/2$
as long as $\varepsilon(\theta^*)$ is small enough.
On the other hand,~\eqref{eq:expectation-decomposed-2} is also smaller
than $2 \eta^{1/4}$ as long as $\varepsilon$ is small enough compared to $C$.
\end{pf*}

\begin{pf*}{Proof of Lemma~\ref{lem:expectation-sqrt-gw}}
Fix some $\varepsilon= \varepsilon(\theta^*) > 0$ to be determined. If
$\theta^2 d$
is sufficiently large compared to $\varepsilon$, Proposition~\ref
{prop:gw-large-expected-mag}
implies that
\begin{eqnarray*}
\Pr(X_{ui,k} \ge1-\varepsilon\mid\sigma_{u} = +) &\ge& 1-
\varepsilon- \Pr(\sigma_{ui} = - \mid\sigma_u = +) \ge1 -
\varepsilon- \eta.
\end{eqnarray*}

Now, if $f$ is any decreasing function then
%
%
\begin{eqnarray}
\label{eq:expectation-sqrt-gw-1} \E f(X) &\le& f(1-\varepsilon) \Pr(X \ge
1-\varepsilon)\nonumber
\\
&&{} + f(0) \Pr(0 \le X < 1-\varepsilon)
\\
&&{} + f(-1) \Pr(X < 0).
\nonumber
\end{eqnarray}
We will apply this with $f(x) = \sqrt{\frac{1-\theta x}{1+\theta
x}}$; note that
$f(0) = 1$ and $f(-1) = \sqrt{(1-\eta)/\eta}$, where $\eta= \frac
{1-\theta}{2}$.

Now, we consider two regimes. If $\sqrt\eta\ge\theta^*/10$, we bound
%
%
\begin{eqnarray}
\label{eq:expectation-sqrt-gw-2} \E\bigl(f(X_{ui,k}) \mid\sigma_u = +
\bigr) &\le&\Pr(X_{ui,k} \ge1-\varepsilon\mid\sigma_u = +) f(1-
\varepsilon)\nonumber
\\
&&{}+ \Pr(X_{ui,k} < 1-\varepsilon\mid\sigma_u = +) f(-1)
\nonumber\\[-8pt]\\[-8pt]\nonumber
&\le&(1-\varepsilon- \eta) f(1-\varepsilon) + \frac{\varepsilon+ \eta
}{\sqrt{\eta}}
\nonumber
\\
&\le&(1 - \eta)f(1-\varepsilon) + \sqrt{\eta(1-\eta)} + \frac
{10\varepsilon}{\theta^*}.
\nonumber
\end{eqnarray}
Now, $f(1-\varepsilon) = \frac{\eta}{1-\eta} + O(\varepsilon)$, and so
\[
\E\bigl(f(X_{ui,k}) \mid\sigma_u = + \bigr) \le2\sqrt{
\eta(1-\eta)} + O(\varepsilon),
\]
where the constants in $O(\varepsilon)$ depend on $\theta^*$. Since
$2\sqrt{\eta(1-\eta)}$ is bounded
away from $1$ while $\eta$ is bounded away from $1/2$, it follows that
for small enough $\varepsilon$
(depending on $\theta^*$), $\E(f(X_{ui,k}) \mid\sigma_u = +)$ is
bounded away from 1.

On the other hand, if $\sqrt\eta\le\theta^* / 10$ then we
use~\eqref{eq:expectation-sqrt-gw-1} and the fact (from Lemma~\ref
{lem:wrong-sign}) that
$\Pr(X_{ui,k} < 0 \mid\sigma_u = +) \le2\eta$ to bound
\begin{eqnarray*}
\E f(X) &\le&(1-\varepsilon) f(1-\varepsilon) + \varepsilon f(0) + 2\eta f(-1)
\\
&\le& f(1-\varepsilon) + \varepsilon+ 2\sqrt{\eta}.
\end{eqnarray*}
Now, if $\varepsilon\le\frac{1}2$ then $f(1-\varepsilon) \le\sqrt
{1-\theta^*/2}
\le1-\theta^*/4$, so
\[
\E f(X) \le1-\theta^*/4 + \varepsilon+ 2\sqrt{\eta} \le1 - \frac
{\theta^*}{20} +
\varepsilon,
\]
which is bounded away from 1 if $\varepsilon$ is small enough.
\end{pf*}
\end{appendix}

\section*{Acknowledgment}
The authors thank Jiaming Xu for his careful reading of the manuscript
and his helpful comments and corrections.


%

\printaddresses
\end{document}